\documentclass[a4paper, reqno]{amsart}

\usepackage{amsmath,amsthm,amssymb}
\usepackage[nobysame, alphabetic]{amsrefs}
\usepackage[margin=30mm]{geometry}
\usepackage{enumerate}
\usepackage{tikz}
\usetikzlibrary{cd, arrows}
\usepackage{tikz-cd}
\usepackage[arrow,matrix,curve,cmtip,ps]{xy}
\usepackage{pb-diagram}
\usepackage{pb-xy}
\usepackage{comment}
\usepackage[normalem]{ulem}
\usepackage{xcolor}

\newcommand{\kervaire}{\mathrm{kerv}}

\definecolor{darkblue}{rgb}{0,0,0.6}

\usepackage[breaklinks, pdftex, ocgcolorlinks,colorlinks=true, citecolor=darkblue, filecolor=darkblue, linkcolor=darkblue, urlcolor=darkblue]{hyperref}
\usepackage[capitalize,nameinlink,noabbrev]{cleveref}
 \usepackage{adjustbox}

\makeatletter
\newtheorem*{rep@theorem}{\rep@title}
\newcommand{\newreptheorem}[2]{%
	\newenvironment{rep#1}[1]{%
		\def\rep@title{#2 \ref{##1}}%
		\begin{rep@theorem}}%
		{\end{rep@theorem}}}
\makeatother

\newtheorem{proposition}{Proposition}[section]
\newtheorem{theorem}[proposition]{Theorem}
\newtheorem{corollary}[proposition]{Corollary}
\newtheorem{lemma}[proposition]{Lemma}

\newreptheorem{theorem}{Theorem}
\newreptheorem{lemma}{Lemma}
\newreptheorem{proposition}{Proposition}
\newreptheorem{corollary}{Corollary}

\newtheorem{thmx}{Theorem}

\theoremstyle{definition}
\newtheorem{definition}[proposition]{Definition}
\newtheorem{question}[proposition]{Question}

\theoremstyle{remark}
\newtheorem{remark}[proposition]{Remark}

\newtheorem*{remark*}{Remark}

\numberwithin{equation}{section}

\DeclareMathOperator{\Diff}{Diff}

\newcommand{\sign}{\operatorname{sign}}

\DeclareMathOperator{\spin}{Spin}
\newcommand{\wh}{\widehat}

\DeclareMathOperator{\tr}{tr}

\newcommand{\N}{\mathbb{N}}

\newcommand{\R}{\mathbb{R}}

\newcommand{\Z}{\mathbb{Z}}

\newcommand{\FF}{\mathbb{F}}

\newcommand{\M}{\mathcal{M}}

\newcommand{\im}{\operatorname{Im}}

\newcommand{\id}{{\operatorname{id}}}

\DeclareMathOperator{\ttop}{top}

\newcommand{\ol}{\overline}
\newcommand{\wt}{\widetilde}

\newcommand{\ks}{\operatorname{ks}}
\newcommand{\CP}{\mathbb{CP}}

\DeclareMathOperator{\Sq}{Sq}
\DeclareMathOperator{\Out}{Out}
\DeclareMathOperator{\Aut}{Aut}
\DeclareMathOperator{\Hom}{Hom}
\DeclareMathOperator{\ev}{ev}
\DeclareMathOperator{\pr}{pr}

\newcommand{\Top}{\mathrm{Top}}
\newcommand{\diff}{\mathrm{diff}}

\newcommand{\Ext}{\mathrm{Ext}}
\newcommand{\G}{\mathrm{G}}

\newcommand{\BSTop}{\mathrm{BSTop}}
\newcommand{\BSO}{\mathrm{BSO}}

\newcommand{\topspin}{\mathrm{TopSpin}}
\newcommand{\cN}{\mathcal{N}}
\newcommand{\ter}{\mathrm{ter}}
\newcommand{\pri}{\mathrm{pri}}
\renewcommand{\sec}{\mathrm{sec}}

\DeclareMathOperator{\st}{st}
\newcommand{\homeo}{\cong_{\text{\normalfont top}}}
\DeclareMathOperator{\cb}{cb}

\newenvironment{clist}[1]
{\begin{enumerate}[\normalfont #1]}
{\end{enumerate}}

\usepackage{caption}
\usepackage{placeins}

\makeatletter
\@namedef{subjclassname@2020}{%
  \textup{2020} Mathematics Subject Classification}
\makeatother


\begin{document}

\title{Stable equivalence relations on 4-manifolds}
 
\author{Daniel Kasprowski}
\address{School of Mathematical Sciences, University of Southampton, Southampton SO17 1BJ, UK}
\email{d.kasprowski@soton.ac.uk}
	
\author{John Nicholson}
\address{School of Mathematics and Statistics, University of Glasgow, U.K.}
\email{john.nicholson@glasgow.ac.uk}
	
\author{Simona Vesel\'a}
\address{Universit\"at Bonn, Regina-Pacis-Weg 3, Bonn, Germany}
\email{vesela@math.uni-bonn.de }
	
\begin{abstract}
Kreck's modified surgery gives an approach to classifying smooth $2n$-manifolds up to stable diffeomorphism, i.e. up to connected sum with copies of $S^n \times S^n$. In dimension 4, we use a combination of modified and classical surgery to study various stable equivalence relations which we compare to stable diffeomorphism. Most importantly, we consider homotopy equivalence up to stabilisation with copies of $S^2 \times S^2$.

As an application, we show that closed oriented homotopy equivalent 4-manifolds with abelian fundamental group are stably diffeomorphic. 
We give analogues of the cancellation theorems of Hambleton--Kreck for homotopy equivalence up to stabilisations. 
Finally, we give a complete algebraic obstruction to the existence of closed smooth 4-manifolds which are homotopy equivalent but not simple homotopy equivalent up to connected sum with $S^2 \times S^2$.
\end{abstract}

\maketitle

\vspace{-8mm}
 
\section{Introduction}

Surgery theory, as developed by Browder,  Novikov, Sullivan, Wall and others, gives a method for determining when two closed $n$-manifolds which are (simple) homotopy equivalent are actually diffeomorphic. These methods work well when $n \ge 5$. For $n=4$, they break down in the smooth category by Donaldson \cite{Do87} but work in the topological category over \emph{good} fundamental groups by Freedman \cites{Freedman82,Freedman-Quinn}.
Kreck's modified surgery \cite{surgeryandduality} is a method to classify manifolds up to the weaker notion of stable diffeomorphism. This applies in dimension 4, where two closed smooth 4-manifolds $M$, $M'$ are \emph{stably diffeomorphic} if there exist $k,k' \ge 0$ and a diffeomorphism 
\[ M\#k(S^2\times S^2)\cong M'\#k'(S^2\times S^2).\]

The aim of this article is to introduce a wider range of stable equivalence relations on 4-manifolds which can be studied using modified surgery. 
By \cite[Theorem C]{surgeryandduality}, two closed oriented smooth 4-manifolds $M, M'$ are stably diffeomorphic if and only if they have the same normal $1$-type $\xi\colon B\to \BSO$ and they admit normal $1$-smoothings $\wt \nu_{M}$, $\wt \nu_{M'}$ such that $(M,\wt \nu_{M})$, $(M',\wt \nu_{M'})$ are equal in the group $\Omega_4(\xi)$ of bordisms over $\xi$ (see \cref{sec:background}).
Teichner \cite[Theorem~3.1.1]{teichnerthesis} constructed a spectral sequence converging to $\Omega_4(\xi)$ which, in the case when $M$ is almost spin, i.e. its universal cover $\wt M$ is spin, has $E^2$-term $E^2_{p,q}=H_p(\pi;\Omega^{\spin}_q)$ where $\pi = \pi_1(M)$.
It induces a filtration
\[16\Z=H_0(\pi;\Omega_4^{\spin})=F_{0,4}\leq F_{2,2}\leq F_{3,1}\leq F_{4,0}=\Omega_4(\xi),\]
where there is no $F_{1,3}$-term since $\Omega_3^{\spin}=0$.
For a subgroup $A\leq \Omega_4(\xi)$ we say that $M, M'$ are \textit{$\xi$-bordant mod $A$} if there exist normal $1$-smoothings $\wt \nu_{M}$, $\wt \nu_{M'}$ such that $[(M,\wt \nu_{M})]-[(M',\wt \nu_{M'})] \in A$. 

Our main result establishes a correspondence between geometrically defined stable equivalence relations and $\xi$-bordism mod $A$. 

\begin{thmx} \label{thmx:main-table}
Let $M$, $M'$ be closed, oriented, almost spin, smooth $4$-manifolds with normal $1$-type $\xi=\xi(\pi,w)$. 
 For each subgroup $A \le \Omega_4(\xi)$ listed below, the manifolds $M$ and $M'$ are $\xi$-bordant mod $A$ if and only if they are related by the geometric equivalence relation on the right.

\bgroup
 \def\x{3.6}
 \def\y{0}
 \vspace{-\x mm}
\def\arraystretch{1.2}
\begin{figure}[h] \label{table:main}
\begin{center}
\setlength{\tabcolsep}{4pt}
\begin{tabular}{rlcp{9.4cm}}
{\normalfont(i)}  & \, $0$ & $\leftrightarrow$& Stable diffeomorphism\\
{\normalfont(ii)}  & \, $F_{0,4}$ & $\leftrightarrow$& There exist simply connected spin $4$-manifolds $L,L'$ \hspace{10mm} such that $M\#L\cong M'\# L'$ \\
{\normalfont(iii)} & \, $[\ker(\kappa_2^s)\cap\ker(w\frown-)]$ & $\leftrightarrow$& Simple homotopy equivalence up to stabilisations by $S^2 \times S^2$ \\
{\normalfont(iv)} & \, $[\ker(\kappa_2^h)\cap\ker(w\frown-)]$ & $\leftrightarrow$& Homotopy equivalence up to stabilisations by $S^2 \times S^2$ \\
{\normalfont(v)} & \, $F_{2,2}$ & $\leftrightarrow$& There exist $k,k'\ge 0$ and a $2$-connected degree one normal map $f\colon M\#k(S^2\times S^2)\to M' \# k' (S^2 \times S^2)$\\
{\normalfont(vi)} & \, $F_{3,1}$ & $\leftrightarrow$& There exist simply connected $4$-manifolds $K,K'$ such \hspace{10mm} that $M\#K\cong M'\# K'$. \\
\end{tabular}
\end{center}
\vspace{-\x mm}
\end{figure}
\egroup  
\end{thmx}
\FloatBarrier

Here $\kappa_2^s, \kappa_2^h \colon H_2(\pi;\Z/2)\to L_4^s(\Z\pi)$ denote the components of the surgery assembly maps and $[\,\cdot\,]$ is the composition $H_2(\pi;\Z/2)\cong E^2_{2,2}\twoheadrightarrow E^{\infty}_{2,2}\to  F_{2,2}\leq \Omega_4(\xi)$ (see \cref{sec:background} for further details). 

\begin{remark}
\begin{clist}{(a)}
\item 
Let $M$, $M'$ be manifolds as in \cref{thmx:main-table} but not almost spin, i.e. the universal covers are not spin. Then several of the geometric equivalence relations agree and otherwise the difference is determined by the signatures. More specifically, we have:
\begin{clist}{(1)}
    \item $M$ and $M'$ are stably diffeomorphic if and only if they are (simple) homotopy equivalent.
    \item There exists a $2$-connected degree one normal map $f\colon M\#k(S^2\times S^2)\to M' \# k' (S^2 \times S^2)$ for some $k,k'\ge 0$  if and only if there exist simply connected spin $4$-manifolds $L,L'$ such that $M\#L\cong M'\# L'$. 
    If so, $M$ and $M'$ are stably diffeomorphic if and only if $\sigma(M)=\sigma(M')$.
    \item If there exist simply connected $4$-manifolds $K,K'$ such that $M\#K\cong M'\# K'$, then $K$ and $K'$ can be chosen to be spin if and only if $\sigma(M)\equiv \sigma(M')\mod 8$.
\end{clist}
This follows from \cite[Theorem~C]{surgeryandduality}, see also \cite[Lemma~2.1]{KPT}.
\item 
We also prove a version of the theorem in the topological category (see \cref{thm:main-top}).
\item 
We can always take $k'=0$ in (v) since the collapse map $M' \# k'(S^2 \times S^2) \to M'$ is a $2$-connected degree one normal map.
It is also not immediately obvious that the geometric relation in (v) is actually an equivalence relation. Whilst this is implied by the statement above, we check this directly in \cref{rem:eq-rel-direct}.
\item 
The geometric interpretations of $F_{0,4}$ and $F_{3,1}$ follow directly from Kreck's modified surgery theory. In fact, the stable equivalence relation corresponding to $F_{3,1}$ coincides with $\CP^2$-stable diffeomorphism which was studied in \cite{KPT}.
In the case of 2-dimensional fundamental groups, a connection between $\ker(\kappa_2)$ and $\Omega_4(\xi)$ was previously established in \cite[Lemma~5.11]{HKT}. The geometric interpretation of $F_{2,2}$ is new. 
\end{clist}
\end{remark}
 
It was shown by Gompf \cite{gompf} that two closed oriented smooth 4-manifolds are stably homeomorphic if and only if they are stably diffeomorphic. 
This equivalence relation also coincides with being stably $h$-cobordant and stably $s$-cobordant by results of Wall \cite{Wa64} and Lawson \cite{La78} (see also \cite[Theorem 3.4]{KPR22}).

If $\equiv$ is an equivalence relation on closed $4$-manifolds, then write $\equiv^{\st}$ for the corresponding stable equivalence relation, i.e. $M \equiv^{\st} M'$ if $M \# a(S^2 \times S^2) \equiv M' \# b(S^2 \times S^2)$ for some $a,b \ge 0$. 
We refer to $\simeq^{\st}$ (resp. $\simeq_s^{\st}$) as \emph{(simple) homotopy equivalent up to stabilisation}.
We have intentionally avoided use of the term `stably homotopy equivalent' in order to avoid confusion with notions from stable homotopy theory. Let $\homeo$ denote homeomorphism. Then $\cong^{\st}$ (resp. $\homeo^{\st}$) denotes stable diffeomorphism (resp. homeomorphism), which coincide for closed oriented smooth $4$-manifolds.

The proof of \cref{thmx:main-table} involves a mixture of techniques from both surgery and modified surgery. The geometric interpretation of $F_{2,2}$ in terms of degree one normal maps is used to establish the interpretations of $[\ker(\kappa_2^s)]$ and $[\ker(\kappa_2^h)]$ (see \cref{s:proofs}). 
 
\subsection{Comparison between stable equivalence relations} \label{ss:comparison-stable}

Stable diffeomorphism has for example been studied in \cites{teichnerthesis,HKT,KLPT15}. From these references it follows that $\xi$-bordism mod $A$ is different for $A=0$, $F_{0,4}$, $F_{2,2}$, $F_{3,1}$ and $F_{4,0}$ (see, for example, \cite[Theorem~4.4.9]{teichnerthesis}).
For the other equivalence relations coming from \cref{thmx:main-table}, we have:
\[
\substack{\text{\normalfont\normalsize stable diffeomorphism ($\cong^{\st}$)} \\ \text{\normalfont\normalsize ($\Leftrightarrow$ stable homeomorphism ($\homeo^{\st}$))}} 
\, \Rightarrow \, 
\substack{\text{\normalfont\normalsize simple homotopy equivalence} \\ \text{\normalfont\normalsize up to stabilisations ($\simeq_s^{\st}$)}} 
\, \Rightarrow \,
\substack{\text{\normalfont\normalsize homotopy equivalence up} \\ \text{\normalfont\normalsize to stabilisations ($\simeq^{\st}$)}}\]
for closed oriented smooth $4$-manifolds.

We will now consider the extent to which these equivalence relations coincide.
In order to study the difference between homotopy equivalence up to stabilisation and stable homeomorphism, we will make the following definition.

	\begin{definition}
		\label{def:stable-rigid}
		A group $\pi$ satisfies \emph{stable rigidity} if any two closed oriented smooth homotopy equivalent $4$-manifolds $M, M'$ with fundamental group $\pi$ are stably diffeomorphic.
	\end{definition}

\begin{remark} \label{remark:implications-stable}
This is equivalent to the same definition with `homotopy equivalent' replaced by `homotopy equivalent after stabilisations'. In particular, $\pi$ satisfies stable rigidity if and only if the three stable equivalence relations coincide for 4-manifolds with fundamental group $\pi$. 
\end{remark}
 
It follows from \cref{thmx:main-table} that if $\kappa_2^h$ is injective for a group $\pi$, then $\pi$ is stably rigid. This recovers a result of Davis~\cite{Davis-05}*{Theorem~1.4}. In general however, $\kappa_2^h$ is not injective and Teichner gave examples of pairs of smooth manifolds with quaternionic \cite[Example~5.2.4]{teichnerthesis} or infinite dihedral fundamental group \cite[Proposition~3]{teichner-star} 
that are homotopy equivalent but not stably diffeomorphic. While $\kappa_2^h$ is not injective for abelian groups in general, we use \cref{thmx:main-table} together with known calculations of the assembly map to show the following.
	
\begin{thmx}
\label{thm:abelian}
If $\pi$ is a finitely generated abelian group, then $\pi$ satisfies stable rigidity. That is, if $M, M'$ are two closed oriented smooth homotopy equivalent $4$-manifolds with abelian fundamental groups, then $M,M'$ are stably diffeomorphic.
\end{thmx}

Equivalently, homotopy equivalence up to stabilisations coincides with stable diffeomorphism for closed oriented smooth $4$-manifolds whose fundamental group is abelian. This can be viewed as a generalisation of the result of Gompf \cite{gompf} stated above.

We will also consider the case where $\pi$ is a finite group whose Sylow $2$-subgroup is cyclic, quaternionic, dihedral, semi-dihedral or modular maximal-cyclic (see \cref{s:stablerigidity} for precise definitions). The first four classes are the so-called \textit{basic groups}. They have special significance in surgery theory as basic subquotients determine surgery obstructions in the $p$-decorated L-groups $L_4^p(\Z\pi)$ (see \cite[Section 2]{HM80} and \cite[Theorem A]{HM93}). We completely determine which of these groups satisfy stable rigidity in \cref{s:stablerigidity}.

We will next consider the difference between simple homotopy equivalence up to stabilisations and homotopy equivalence up to stabilisations. As in \cref{remark:implications-stable}, we would equivalently like to know whether homotopy equivalence implies simple homotopy equivalence up to stabilisations.
The following reduces the existence of such examples to a problem concerning $\kappa_2^s$ and $\kappa_2^h$.

        \begin{thmx}
	\label{cor:he-vs-she}
	There exist closed oriented smooth $4$-manifolds $M$, $M'$ that are homotopy equivalent but not simple homotopy equivalent up to stabilisations if and only if there exists a finitely presented group $\pi$ with $\ker(\kappa_2^s)\neq\ker(\kappa_2^h) \subseteq H_2(\pi;\Z/2)$.
	\end{thmx}

    \begin{remark}
    If $\ker(\kappa_2^s)\neq\ker(\kappa_2^h)$ for a group $\pi$, then the closed oriented smooth $4$-manifolds $M$, $M'$ constructed in the proof take a very particular form and have fundamental group $\pi \ast \pi$.
It might not always be possible to find $M$, $M'$ with fundamental group $\pi$.
This is because to apply \cref{thmx:main-table} it does not suffice to have $\ker(\kappa_2^s)\neq\ker(\kappa_2^h)$ but we need the existence of a $w\in H^2(\pi;\Z/2)$ such that 
        \[[\ker(\kappa_2^s)\cap \ker(w\frown -)]\neq [\ker(\kappa_2^h)\cap \ker(w\frown -)]\leq \Omega_4(\xi(\pi,w)).\]
        On the other hand, the proof can be used to construct topological $4$-manifolds with fundamental group $\pi$ that  are homotopy equivalent but not simple homotopy equivalent up to stabilisations.
    \end{remark}

Examples of closed topological $4$-manifolds which are homotopy equivalent but not simple homotopy equivalent were found only recently \cite{NNP23}. It is currently open whether or not there exist closed smooth $4$-manifolds which are homotopy equivalent but not simple homotopy equivalent. The above shows that examples exist provided there exists a group $\pi$ with $\ker(\kappa_2^s)\neq\ker(\kappa_2^h)$. We therefore ask the following question.

\begin{question}
Does there exist a finitely presented group $\pi$ such that $\ker(\kappa_2^s)\neq\ker(\kappa_2^h)$?\footnote{
    While this paper was under revision, the authors were informed of upcoming work of Hambleton and \"{U}nl\"{u} \cite{HU25} showing the existence of finitely presented groups for which $\ker(\kappa_2^s) \neq \ker(\kappa_2^h)$. Hence \cref{cor:he-vs-she} shows the existence of closed smooth $4$-manifolds which are homotopy equivalent but not simple homotopy equivalent (even up to stabilisations).
}
\end{question}

We resolve this negatively for several classes of groups and thereby obtain the following result.

\begin{theorem} \label{prop:k_2-equal}
Let $\pi$ be a finitely generated abelian group or a finite group whose Sylow $2$-subgroup is abelian, basic, modular maximal-cyclic, or of order at most $16$.
Then $\ker(\kappa_2^s) = \ker(\kappa_2^h)$. In particular, closed, oriented, smooth $4$-manifolds $M$, $M'$ with fundamental group $\pi$ which are homotopy equivalent are simple homotopy equivalent up to stabilisations.
\end{theorem}

\subsection{Comparison between stable and unstable equivalence relations} \label{ss:stable-vs-unstable}

Kreck's modified surgery makes it possible to classify closed oriented smooth $4$-manifolds up to stable diffeomorphism (or equivalently stable homeomorphism) over special choices of fundamental group $\pi$. By \cref{thmx:main-table}, the same can be said for (simple) homotopy up to stabilisations (see \cref{ss:comparison-stable}).

In order to complete the classification up to homeomorphism (resp. (simple) homotopy equivalence), it remains to solve the corresponding \textit{cancellation problem}.
In the case of homeomorphism, which we denote by $\homeo$, this asks for conditions under which
\[ M \# (S^2 \times S^2) \homeo N \# (S^2 \times S^2) \quad \Rightarrow \quad M \homeo N.\]
This was studied by Hambleton-Kreck \cite{hambleton-kreck93} and Crowley-Sixt \cite{CS}.

The following allows us study the cancellation problem up to (simple) homotopy equivalence and compare it to the homeomorphism case.

\begin{thmx}
\label{thm:decomp}
Let $M$, $N$ be closed oriented topological $4$-manifolds with good fundamental groups. If $M$, $N$ are (simple) homotopy equivalent up to stabilisations, then there exists a topological $4$-manifold $N'$ which is (simple) homotopy equivalent to $N$ and stably homeomorphic to $M$.
\end{thmx}

\begin{remark}
   Homotopy equivalence up to stabilisations coincides with the equivalence relation generated by homotopy equivalence and stable homeomorphism. \cref{thm:decomp} shows that, on the class of closed topological $4$-manifolds with good fundamental group, only a single homotopy equivalence and a single stable homeomorphism are needed.
\end{remark}

The proof is based on the observation that, if $M$, $N$ are closed oriented topological $4$-manifolds with good fundamental groups such that $M \simeq N \# (S^2 \times S^2)$, then there exists a topological $4$-manifold $N'$ such that $M \homeo N' \# (S^2 \times S^2)$ and $N \simeq N'$ (see \cref{prop:stable-summand}). 
This is shown to be a consequence of Freedman's disc embedding theorem \cite{Fr83} (see also \cite[Theorem 2.3]{PRT21}).

It was shown by Hambleton-Kreck that, if two closed oriented topological 4-manifolds $M$, $N$ with finite fundamental group are stably homeomorphic and have the same Euler characteristic, then $M \# (S^2 \times S^2) \homeo N \# (S^2 \times S^2)$ \cite[Theorem B]{hambleton-kreck93}. 
Combining this theorem with \cref{thm:decomp}, we obtain the following homotopy cancellation result.

\begin{corollary}
    \label{thm:cancellation}
	Let $M$, $N$ be closed oriented topological $4$-manifolds with the same Euler characteristic and finite fundamental group which are (simple) homotopy equivalent up to stabilisations. If $M\simeq M_0\# (S^2 \times S^2)$ for a topological $4$-manifold $M_0$, then $M$, $N$ are (simple) homotopy equivalent.
\end{corollary}

In fact, we prove a much more general statement about the relationship between the cancellation problems for homeomorphism and (simple) homotopy equivalence (see \cref{thm:main-cancellation}).

\subsection*{Organisation of the paper}
In \cref{sec:background}, we recall the basic definitions of Kreck's modified surgery \cite{surgeryandduality}, in particular the r\^{o}le the $\xi$-bordism group $\Omega_4(\xi)$ as well as classical surgery. We recall the description of \cite{Davis-05} (see also \cite{HMTW88}) of the surgery obstruction map.
Our main technical tools are given in \cref{thm:davis} (due to Davis) and \cref{thm:davis-general}.

In \cref{s:proofs}, we restate \cref{thmx:main-table} in topological category (\cref{thm:main-top}) and give a proof. To obtain \cref{thmx:main-table}, we compare the topological (resp. smooth) bordism group $\Omega_4(\xi^{\ttop})$ (resp. $\Omega_4(\xi^\diff)$) and identify the Kirby-Siebenmann invariant in $\Omega_4(\xi^{\ttop})$ (\cref{lem:ter-ks}).

In \cref{s:stablerigidity}, we define (strong) stable rigidity for a group $\pi$. 
When $\pi$ is finite, these properties depend only of the $2$-Sylow subgroup of $\pi$ (\cref{prop:oddindex}).
We prove \cref{thm:abelian} and determine stable rigidity for quaternion (\cref{ss:quaternion}), dihedral or semi-dihedral groups (\cref{ss:Dih_SemiDih}).

In \cref{s:he-vs-she}, we algebraically describe the difference between homotopy and simple homotopy equivalence up to stabilisations (\cref{cor:he-vs-she}).
We show that these equivalence relations are the same if $\pi$ is finite with 2-Sylow subgroups of order $\leq 16$ (\cref{prop:k_2-equal}).

In \cref{sec:cancellation}, we establish \cref{thm:decomp} and use it to find an explicit relationship between the cancellation problems for $4$-manifolds up to homeomorphism and (stable) homotopy (\cref{thm:main-cancellation}).

\subsection*{Conventions}  
Since the results proven are more general, we will work in the topological category for the remainder of this article. Unless otherwise stated, all manifolds are assumed to be closed, connected and oriented.
All groups will be assumed to be finitely presented.
 
\subsection*{Acknowledgements}
JN was supported by the Heilbronn Institute for Mathematical Research as well as a Rankin-Sneddon Research Fellowship and a  John Robertson Bequest Award from the University of Glasgow. 
SV was supported by the Deutsche Forschungsgemeinschaft (DFG, German Research Foundation) under Germany’s Excellence Strategy – EXC-2047/1 – 390685813. The authors thank the University of Glasgow, the University of Southampton and the Max Planck Institute for Mathematics in Bonn for hospitality. We would like to thank Ian Hambleton, Mark Powell and Peter Teichner for useful conversations. We would also like to thank Werner Bley and Graham Ellis for useful discussions on computations.

\section{Surgery theory for 4-manifolds}\label{sec:background}

In preparation for the proof of \cref{thmx:main-table}, we will now recall methods from Kreck's modified surgery (\cref{ss:surgery-kreck}) and classical surgery (\cref{ss:surgery-classical}). In \cref{subsec:transfer_L}, we include a result on $L$-groups $L_4(\Z \pi)$ which we will make use of in \cref{s:stablerigidity}.

While the results in the introduction were stated in the smooth category, we will work in the topological category from now on. Since two closed, oriented, smooth 4-manifolds are stably homeomorphic if and only if they are stably diffeomorphic \cite{gompf}, the results in the topological category imply the analogous results in the smooth category.

	\begin{remark}		
Since it is not currently known whether topological $4$-manifolds are homeomorphic to CW-complexes, the usual definition of simple homotopy equivalence for CW-complexes may not apply. A more general notion of simple homotopy equivalence, which applies to topological $4$-manifolds, has been defined by Kirby-Siebenmann \cite[III, \S 4]{KirbySiebenmann}. The idea is to embed both manifolds into a high dimensional Euclidean space, and say that a homotopy equivalence $f\colon M\to N$ is simple if and only if the induced map between the induced normal disk bundles is a simple homotopy equivalence (see also \cite[Section 2.1]{NNP23}). Throughout this article, we will take this to be our definition of simple homotopy equivalence.
	\end{remark}

\subsection{Kreck's modified surgery}
\label{ss:surgery-kreck}
 
	We start this section by reviewing Kreck's stable classification result. While \cite[Theorem~C]{surgeryandduality} holds for all even dimensional manifolds, we restrict to manifolds of dimension 4. As before, all manifolds are assumed to be closed, connected and orientable. 
	
	Recall that a map between connected spaces is called $2$-connected if it induces an isomorphism on $\pi_1$ and a surjection on $\pi_2$. A map between connected spaces is $2$-coconnected if it induces an isomorphism on $\pi_k$ for $k\geq 3$ and an injection on $\pi_2$.

By $\BSTop$ we mean the classifying space of stable orientable microbundles or equivalently stable orientable fiber bundles with fibre some $\R^n$. This is in analogy with $\BSO$ the classifying space of orientable stable vector bundles. An oriented topological manifold $M$ naturally admits a stable \emph{normal microbundle} 
$\nu_M \colon M\to \BSTop$, 
see \cite{hirsch-microbundles} and \cite{KirbySiebenmann}*{IV, Appendix A}.
 
 \begin{definition}
		Let $M$ be a $4$-manifold. Its \emph{normal $1$-type} is a $2$-coconnected fibration $\xi\colon B\to \BSTop$ such that the stable normal bundle $\nu_M\colon M\to \BSTop$ of $M$ admits a $2$-connected lift $\wt \nu_M\colon M\to B$ along $\xi$. Any such lift $\wt\nu_M$ is a \emph{normal $1$-smoothing} of $M$.
	\end{definition}
\begin{remark}
    For every $4$-manifold $M$ there exists a normal $1$-type, which is unique up to homotopy equivalence over $\BSTop$, and a normal $1$-smoothing. This follows directly from the existence and uniqueness of Moore--Postnikov sections \cite{hatcher}*{Theorem~4.71} applied to the stable normal bundle $M\to \BSTop$.
\end{remark}
	\begin{definition}
		Let $\xi\colon B\to \BSTop$ be a fibration. Then the bordism group $\Omega_n(\xi)$ consists of classes $[M,\wt\nu_M]$ of $n$-manifolds $M$ with lifts $\wt\nu_M\colon M\to B$ of their stable normal bundle along $\xi$ up to bordism over $\xi$. Here, a bordism between $(M,\wt\nu_M)$ and $(N,\wt\nu_M)$ consists of a bordism $W$ between $M$ and $N$ and a lift $\wt\nu_W\colon W\to B$ of the stable normal bundle of $W$ that restrict to the lifts $\wt\nu_M$ and $\wt\nu_N$, respectively.
	\end{definition}
 
	\begin{remark}
Denote by $\Aut(\xi)$ the group of self homotopy equivalences $B\to B$ over $\xi$. It acts transitively on the homotopy classes of normal $1$-smoothings of $M$. The group $\Aut(\xi)$ acts on $\Omega_4(\xi)$ by changing the $\xi$-structures. More details about this action will be given in \cref{rm:aboutXiActionOnOmega}.
	\end{remark}

    The following is a theorem by Kreck.
    
 \begin{theorem}[{\cite[Theorem C]{surgeryandduality}}]
		Two $4$-manifolds with normal $1$-type $\xi$ are stably homeomorphic if and only if their normal $1$-smoothings represent the same class in $\Omega_4(\xi)/\Aut(\xi)$.
	\end{theorem}
 
	We now give an overview over the possible normal $1$-types of $4$-manifolds. For a more detailed discussion of this, see for example \cite[Section~4]{KLPT15}.
	Let $M$ be a manifold with $\pi_1(M)=\pi$. If $\wt M$ is not spin, then the normal $1$-type of $M$ is $\xi:=\pr_2\colon B\pi_1(M)\times\BSTop\to \BSTop$ and $\Aut(\xi)=\Out(\pi_1(M))$. Two $4$-manifolds with normal $1$-type $\pr_2\colon B\pi\times\BSTop\to \BSTop$ are stably homeomorphic if and only if they have the same signature, the same Kirby--Siebenmann invariant and the images of their fundamental classes in $H_4(\pi;\Z)/\Out(\pi)$ agree.
 
    To define the normal $1$-type for almost spin manifolds we will make use of the following construction.

\begin{definition}\label{def:Normal_1_type}
		Let $X$ be a connected space and $w\in H^2(X;\Z/2)$. Then the fibration $\xi(X,w)$ is defined by the following homotopy pullback:
		\[\begin{tikzcd}
			B_{\xi(X,w)}\ar[r,"{p(X,w)}"]\ar[d,"{\xi(X,w)}"']&X\ar[d,"w"]\\			\BSTop\ar[r,"w_2"]&K(\Z/2,2)		
		\end{tikzcd}\] 
  Note that $\xi(X,w)$ is unique up to homotopy equivalence over $\BSTop$.
	\end{definition}
 
	If $M$ is almost spin, i.e.\ $w_2(\wt M)=0$, and has fundamental group $\pi$, then  there exists a unique $w\in H^2(B\pi;\Z/2)$ that pulls back to $w_2(M)$ under the canonical map $c\colon M\to B\pi$. 
In particular, in this case the normal $1$-type $\xi$ is given by $\xi = \xi(\pi,w):=\xi(B\pi,w)$.
	
	There is a James spectral sequence with $E^2$-page $E^2_{p,q}(X,w)=H_p(X;\Omega_q^{\topspin})$ converging to $\Omega_{p+q}(\xi(X,w))$, see \cite[Theorem~3.1.1]{teichnerthesis}. While $E^2_{p,q}(X,w)$ only depends on $X$, the differentials in the spectral sequence also depend on $w$. 
	The James spectral sequence induces a filtration
	\[8\Z=H_0(X;\Omega_4^\topspin)=F_{0,4}(X,w)\leq F_{2,2}(X,w)\leq F_{3,1}(X,w)\leq F_{4,0}(X,w)=\Omega_4(\xi(X,w)),\]
	where there is no $F_{1,3}(X,w)$-term since $\Omega_3^\topspin=0$.

 The second differential of the James spectral sequence was computed in \cite[Theorem 3.1.3]{teichnerthesis}. Namely, let $\Sq^2_w=\Sq^2+-\smile w$. Then
 \begin{itemize}
     \item For $p\leq 4$, the differential $d_2\colon H_p(\pi;\Omega^{\spin}_1)\to H_{p-2}(\pi;\Omega^{\spin}_2)$ is the hom-dual of $\Sq^2_w$.
     \item For $p\leq 5$, the differential $d_2\colon H_p(\pi;\Omega^{\spin}_0)\to H_{p-2}(\pi;\Omega^{\spin}_1)$ is the mod $2$ reduction composed with the hom-dual of $\Sq^2_w$.
\end{itemize}
For more information on the differentials, see also \cite[Corollary~5.7]{OP-MCG}.

 Let $M$ be a manifold with fundamental group $\pi=\pi_1(M)$, $w_2(\wt{M})=0$, some $w\in H^2(\pi,\Z/2)$ and a $1$-smoothing $\wt{\nu}_M$ into $\xi=\xi(\pi,w)$. We now recall the definition of primary, secondary and tertiary invariants ($\pri,\sec,\ter$) introduced in \cite{teichnerthesis}, see also \cite{KPT21}. Define $\pri(M)$ to be the image of $[M]$ under the edge homomorphism $F_{4,0}(\xi)\to E^{\infty}_{4,0}(\xi)$. If $\pri(M)=0$, we know that  $[M]\in F_{3,1}(\xi)$ and define $\sec(M)$ to be the image of $[M]$ under the homomorphism $F_{3,1}(\xi)\to E^{\infty}_{3,1}(\xi)$. Similarly if $\pri(M)=0=\sec(M)$ define $\ter(M)$ to be the image of $[M]$ under the homomorphism $F_{2,2}(\xi)\to E^{\infty}_{2,2}(\xi)$. In \cite{KPT21}, $\pri,\sec$ and $\ter$ are identified with algebraic data about manifolds in the spin case for many fundamental groups.

Now we recall some facts about the action of $\Aut(\xi)$ on the bordism group $\Omega_4(\xi)$.

\begin{remark}\label{rm:aboutXiActionOnOmega}
    For $\xi=\xi(\pi,w)$ denote by $\Aut(\xi)$ the group of self homotopy equivalences $B\to B$ over $\xi$. In particular, we have a homomorphism $\Aut(\xi) \to \Out(\pi)_w$ where the latter is the group of outer homomorphisms of $\pi$ fixing the class $w$. Note that $\Aut(\xi)$ acts on $\Omega_4(\xi)$. By \cite[Corollary 3.1.2]{teichnerthesis}, we know that the action of $\Aut(\xi)$ on every page $E^k_{*,*}$ factors through $\Out(\pi)_w$. However, this does not mean that the same is true for the action of $\Aut(\xi)$ on the filtration of $\Omega_*(\xi)$:
    \[\Omega_4^{\topspin}=F_{0,4}\leq  F_{2,2} \leq F_{3,1}\leq F_{4,0}=\Omega_4(\xi).\]
    
    As a corollary we see that, for any almost spin or spin manifold $M$ with a $1$-smoothing $\wt\nu_M$ to its normal $1$-type $\xi=\xi(\pi,w)$, we have that $\Aut(\xi)$ acts on $\pri(M)$  through $\Out(\pi)_w$. If $\pri(M)=0$ the same is true for $\sec(M)$ and finally if $\pri(M)=\sec(M)=0$ the same is true for $\ter(M)$.

    As another consequence, assume that the images of the kernels of $\kappa_2^h$ and $\kappa_2^s$ (see below for a definition)
    in $E^{\infty}_{2,2}$ differ  for some group $\pi$. Then after quotienting $F_{2,2}$ with $\Aut(\xi)$ they will still differ. This is because both $\kappa_2^h$ and $\kappa_2^s$ are natural with respect to group homomorphisms. 
    \end{remark}

\subsection{Methods from classical surgery}
\label{ss:surgery-classical}

	We now review some aspects of classical surgery theory. For details see \cite{Browder,Wall,luecksurgery}.
 
	\begin{definition}
		A degree one normal map is a pair $(f,\overline{f})$ consisting of a map $f\colon M \to N$ between manifolds such that $f_*[M] = [N],$ together with a $\Top$-bundle $\xi$ over $N$ and a lift $\overline{f}\colon \nu M\to \xi$ of $f$.
  We denote by $\cN(N)$ the group of bordism classes of degree one normal maps to $N$.
	\end{definition}

 The \emph{structure set} $\mathcal{S}^{h}(N)$ is defined as a set of manifolds with a homotopy equivalence to $N$ up to $h$-cobordism. Similarly we define $\mathcal{S}^{s}(N)$ using simple homotopy equivalences and $s$-cobordisms. For each $X \in \{h,s\}$, there is an obvious map 
 \[\eta\colon \mathcal{S}^X(N)\to \cN(N).\] 
	The group $\cN(N)$ of normal invariants can be identified with $[N,\G/\Top]$ (see \cite{luecksurgery}*{Theorem~7.10 and Remark~7.12}). There exists a $5$-connected map $\G/\Top \to K(\Z,4) \times K(\Z/2,2)$~\citelist{\cite{Madsen-Milgram}\cite{kirby-taylor}*{p.\ 397}}. For a closed $2$-manifold $P$, not necessarily orientable, the bijection $[P ,\G/\Top] \to\Z/2$ is called the Kervaire invariant. 
	Now let $N$ be
	an orientable 
	4-manifold. Then we have the following identifications
	\[\cN(N)\cong [N,G/\Top]\cong H^4(N;\Z)\oplus H^2(N;\Z/2)\cong \Z\oplus H^2(N;\Z/2).\]
	The map $\cN(N)$ to $\Z$ is given by $(\sigma(M)-\sigma(N))/8$. The map to $H^2(N;\Z/2)$ corresponds geometrically to assigning to $f$ the cohomology class given by representing a homology class by a $2$-dimensional submanifold $P$ of $N$, making $f$ transverse to $P$ and assigning to $[P]\in H_2(N;\Z/2)$ the classical Kervaire invariant of the $2$-dimensional normal map $f^{-1}(P)\to P$.
 
	\begin{definition}
		Following \cite{Davis-05}*{Definition~3.5}, we call the image $\kervaire(f)\in H^2(N;\Z/2)$ of $f\in\cN(N)$ the codimension 2
		Kervaire invariant.
	\end{definition}

	Let $\pi$ be the fundamental group of an $n$-manifold $N$.
	The surgery obstruction is a map $\theta^h\colon [N,\G/\Top]\to L^h_n(\Z\pi)$. In dimensions higher then $4$, the surgery obstruction of $f$ is trivial if and only if $f$ is normally bordant to a homotopy equivalence \cite{Wall,KirbySiebenmann}. In dimension $4$ this is only known for certain classes of fundamental groups \cite{Freedman82,Freedman-Quinn,KQ,FT}.
	
	We will later use the following two implications from high-dimensional surgery theory which continue to hold in dimension 4. Firstly, if $f\colon M\to N$ is a homotopy equivalence, we can view $f$ as a degree one normal map and then $\theta^h(f)=0$. Secondly, for each $x\in \ker(\theta^h)$, there exist $k\in\N$ and a homotopy equivalence $f\colon M\to N\#k(S^2\times S^2)$ such that $p_*(f)=x$, where $p\colon N\# k(S^2\times S^2)\to N$ is the collapse map. Here we use that the collapse map induces a map $p_*\colon \cN(N\#k(S^2\times S^2))\to \cN(N)$ sending $f$ to $p\circ f$. The second implication follows from the stable surgery exact sequence, see \cite{kirby-taylor}*{Theorem~4}.

We use the following proposition to relate $\theta^h$ to a map $\kappa_2^h\colon H_2(\pi;\Z/2)\to L_4^h(\Z\pi)$.
The following is \cite{Davis-05}*{Proposition~3.6} (see also \cite{HMTW88}*{Theorem~A} for the case where $\pi$ is finite).
 
	\begin{proposition}
 \label{prop:davis-theta}
		Let $\pi$ be a group, let $N$ be a $4$-manifold with fundamental group $\pi$ and let $c\colon N\to B\pi$ be the identity on $\pi_1$. There are homomorphisms
		\[\kappa_2^h \colon H_2(\pi;\Z/2)\to L_4^h(\Z\pi), \quad I_0\colon \Z\to L_4^h(\Z\pi)\]
		so that for any degree one normal map $f\colon M\to N$ between $4$-manifolds represented by $\wh f \colon N \to \G/\Top$, one has the following characteristic class formulae:
		\[I_0((\sigma(M)-\sigma(N))/8) + \kappa_2^h(c_*(\kervaire(f)\frown[N]))=\theta^h(\wh f).\]
	\end{proposition}
 
Similarly, we have the simple surgery obstruction $\theta^s\colon[N,\G/\Top]\to L^s_4(\Z\pi)$, which, in higher dimensions, measures whether a map is normally bordant to a simple homotopy equivalence. The map $\kappa_2^h$ factors as
\[
\kappa_2^h \colon H_2(\pi;\Z/2)\xrightarrow{\kappa_2^s}L_4^s(\Z\pi)\to L_4^h(\Z\pi).
\]
As for $\theta^s$, we have $I_0((\sigma(M)-\sigma(N))/8) + \kappa_2^s(c_*(\kervaire(f)\frown[N]))=\theta^s(\wh f)$. 
When the group $\pi$ is not clear from the context, we will write $\kappa_2^h = \kappa_2^h(\pi)$ and $\kappa_2^s = \kappa_2^s(\pi)$.

The following theorem of Davis will be the key ingredient in our proof of \cref{thmx:main-table} since it allows us to relate $\kappa_2^h$ to Kreck's modified surgery.

\begin{theorem}[{\cite[Theorem~3.12]{Davis-05}}]
\label{thm:davis}
Let $f\colon M\to N$ be a $2$-connected degree one normal map between closed, oriented, almost spin, smooth $4$-manifolds with the same signature and let $p(N,w_2(N)) \colon B_{\xi(N,w_2(N))} \to N$ be the map defined in \cref{def:Normal_1_type}. Then there are normal $1$-smoothings $\wt\nu_N$ and $\wt\nu_M$ in $\xi:=\xi(N,w_2(N))$ 
such that
\begin{clist}{(i)}
    \item $p(N,w_2(N))\circ\wt\nu_N=\id_N$,
    \item 
    $p(N,w_2(N))\circ\wt\nu_M=f$,
    \item $\alpha:=[M,\wt\nu_M]-[N,\wt\nu_N]$ is in the filtration subgroup $F_{2,2}(N,w_2(N))$ of the James spectral sequence, and
    \item $\alpha$ maps to $\kervaire(f)\frown [N]$ in $E^\infty_{2,2}(N,w_2(N))=H_2(N;\Z/2)$.
\end{clist}
An analogous statement is true in the topological category provided that, in addition, $\ks(M) = \ks(N)$.
\end{theorem}

We note that (i) and (ii) of \cref{thm:davis} can be deduced from the proof in \cite{Davis-05}.

\begin{remark}
Let $f\colon M\to N$ be a degree one normal map which is an isomorphism of fundamental groups. Then $f$ is $2$-connected: for $x\in \pi_2(N)\cong H_2(N;\Z\pi)$ the element $f^*(PD^{-1}(x))\frown [N]$ is a preimage of $x$ under $f_*$. 
\end{remark}

 We note that in \cref{thm:davis} $E^{\infty}_{2,2,}=E^2_{2,2}=H_2(N;\Z/2)$ because all differentials with target $E^k_{2,2}$ vanish for trivial reasons unless $k=2$ in which case it vanishes using the calculation $d_2=(\Sq^2+-\smile w_2(N))^*$.
 
	The statement of the above theorem is also true without assuming that $M$ and $N$ have the same signature and Kirby--Siebenmann invariant and that the normal $1$-smoothing of $M$ is unique given a fixed normal $1$-smoothing on $N$, i.e.\ we show the following.
 
	\begin{theorem}
		\label{thm:davis-general}
		Let $f\colon M\to N$ be a $2$-connected degree one normal map between closed, oriented, almost spin $4$-manifolds. Let $\wt\nu_N$ be a normal $1$-smoothing on $N$ in $\xi=\xi(N,w_2(N))$ with $p(N,w_2(N))\circ \wt\nu_N=\id_N$. Then there is a unique normal $1$-smoothing $\wt\nu_M$ with $p(N,w_2(N))\circ \wt\nu_M=f$ such that $\alpha:=[M,\wt\nu_M]-[N,\wt\nu_N]$ is in the filtration subgroup $F_{2,2}(N,w_2(N))$ of the James spectral sequence and $\alpha$ maps to $\kervaire(f)\frown [N]$ in $E^\infty_{2,2}(N,w_2(N))=H_2(N;\Z/2)$.	
	\end{theorem}
 
\begin{proof}
    We can assume that there exists a point $n\in N$ with a neighbourhood $U$ such that $f|_{f^{-1}(U)}\colon f^{-1}(U)\to U$ is a homeomorphism.
	Consider the degree one normal map $f\#\ol{f}\colon M\#\ol{M}\to N\#\ol{N}$, where the connected sum is taking in $f^{-1}(U)$ and $U$, and $\overline{M}$ and $\overline{N}$ denote the manifolds $M$ and $N$ with opposite orientation, respectively. Now $\sign(M\#\ol{M})=0=\sign(N\#\ol{N})$ and $\ks(M\#\ol{M})=0=\ks(N\#\ol{N})$. Hence by \cref{thm:davis} there are normal $1$-smoothings $\wt{\nu}_1$  of $M\#\ol{M}$ and $\wt{\nu}_2$ of $N\#\ol{N}$ in $\xi':=\xi(N\#\ol{N},w_2(\nu(N\#\ol{N})))$ such that 
\[
     p(N\#\ol{N},w_2(N\#\ol{N}))\circ \wt{\nu}_1=f\#\ol{f},\quad
     p(N\#\ol{N},w_2(N\#\ol{N}))\circ \wt{\nu}_2=\id_{N\#\ol{N}},
\] 
$\beta:=[M\#\ol{M},\wt{\nu}_1]-[N\#\ol{N},\wt{\nu}_2]$ is in the filtration subgroup $F_{2,2}(\xi')$, and $\beta$ maps to
\begin{align*} \kervaire(f\#\ol{f})\frown[N\#\ol{N}] &=(\kervaire(f)\frown[N],\kervaire(\ol{f})\frown\ol{N}) \\
& \in H_2(N\#\ol{N};\Z/2)\cong H_2(N;\Z/2)\oplus H_2(\ol{N};\Z/2).\end{align*}

Let $\xi'':=\xi(N^{(3)},w_2(N))$ and $\ol{\xi}'':=\xi(\ol{N}^{(3)},w_2(\ol{N}))$, where we choose a handle structure of $N$ with a single top handle, i.e.\ $N^{(3)}=N\setminus D^4$. There are natural maps 
\[ 
i_3\colon \Omega_4(\xi'')\to \Omega_4(\xi(N,w_2(N)),\quad 
\ol{i_3}\colon \Omega_4(\ol{\xi}'')\to \Omega_4(\xi(\ol{N},w_2(\ol{N})),\quad
j\colon \Omega_4(\xi'')\oplus\Omega_4(\ol{\xi}'')\to\Omega_4(\xi').
\]
Let $\alpha:=[M,\wt{\nu}_1|_M]-[N,\wt{\nu}_2|_{N}]\in \Omega_4(\xi(N,w_2(N)))$, $\ol{\alpha}:=[\ol{M},\wt{\nu}_1|_{\ol{M}}]-[\ol{N},\wt{\nu}_2|_{\ol{N}}]\in \Omega_4(\xi(\ol{N},w_2(\ol{N})))$, where we use that $\ol{\nu_i}$ is trivial over the separating sphere. 

The class $\alpha$ is represented by $[M\# \ol{N}, \wt{\nu}_1|_{\ol{M}}\#-\wt{\nu}_2|_{N}]$, where the connected sum is taking in a neighborhood around $n$ such that $M\#\ol N$ maps to $N^{(3)}$. 
Hence there is a preimage of $(\alpha,\ol{\alpha})$ in $\Omega_4(\xi'')\oplus\Omega_4(\ol{\xi}'')$ under the map $i_3\oplus \ol{i_3}$.
By construction, this maps to $\beta\in \Omega_4(\xi')$ since both classes are represented by $[M\#\ol M\# \ol N\# N,\wt\nu_1-\wt\nu_2]$. Since $H_5(N\#\ol{N};\Z)=0$ and $H_3(N^{(3)};\Z/2)\to H_3(N\#\ol{N};\Z/2)$ is injective, it follows that $\alpha$ is in the filtration subgroup $F_{2,2}(N,w_2(N))$ of the James spectral sequence. Furthermore, $i_3\oplus\ol{i_3}$ and $j$ restrict to isomorphims on the $F_{2,2}$ filtration steps. Since $\beta$ maps to $\kervaire(f\#\ol{f})\frown[N\#\ol{N}]$, it follows that $\alpha$ maps to $\kervaire(f)\frown [N]$ in $E^\infty_{2,2}(N,w_2(N))=H_2(N;\Z/2)$.

Finally we note that the group $\Aut(\xi)$ has a transitive action on the $\xi$-structures on $N$ and so we can assume that $\wt{\nu}_2|_{N}=\tilde{\nu}_N$. 
For the uniqueness of $\wt{\nu}_M$ we argue as follows. Different normal $1$-smoothings $v_M$ of $M$ in $\xi$ with $p(N,w_2(N))\circ v_M=f$ are in 1-to-1 correspondence with null-homotopies of the map $M\xrightarrow{f} N\xrightarrow{w_2} K(\Z/2,2)$ which are a torsor over $H^1(M;\Z/2)$. This can be viewed as changing the topological spin structure on $f^*\nu_N\oplus\nu_M$. Since $f$ is an isomorphism on $\pi_1$, the action is the same as the action of $H^1(N;\Z/2)$ on $\Omega_4(\xi)$. The action of $x\in H^1(N;\Z/2)$ changes the image of $v_M$ in $E_{3,1}^\infty=H_3(N;\Z/2)$ by $x\frown [N]$ by \cite[Proposition~3.2.4]{teichnerthesis}. Hence there can be only be one normal $1$-smoothing with trivial image in $E_{3,1}^\infty$. It follows that $\wt{\nu}_M$ is unique as claimed.
	\end{proof}
 
	\begin{corollary}
		\label{cor:davis}
		Let $N$ be a closed, oriented, almost spin $4$-manifold, $\xi:=\xi(N,w_2(N))$ and $\wt\nu_N$ a normal $1$-smoothing. Let $\pi:=\pi_1 N$ and let $c\colon N\to B\pi$ be the identity on $\pi_1$.
		
		There is an isomorphism 
		\[V\colon \cN(N)\to F_{2,2}(N,w_2(N))\] sending $f\colon M\to N$ to $[M,\wt\nu_M]-[N,\wt\nu_N]$. Then the surgery obstruction map factors as
		\[\cN(N)\xrightarrow{V}F_{2,2}(N,w_2(N))\stackrel{(\sigma/8, \ter)}{\cong}\Z\oplus H_2(N;\Z/2)\xrightarrow{(\id,c_*)}\Z\oplus H_2(\pi;\Z/2)\xrightarrow{I_0+\kappa^s_2}L^s_4(\Z\pi)\to L_4^h(\Z\pi).\]
	\end{corollary}
 
	\begin{proof}
		Using surgery, any normal bordism class has a $2$-connected representative. Hence the map $V$ is well-defined by \cref{thm:davis-general}. Using that $\cN(N)\cong \Z\oplus H^2(N;\Z/2)$ given by sending $f\colon M\to N$ to $(\tfrac{\sigma(M)-\sigma(N)}{8},\kervaire(f))$, we see from \cref{thm:davis-general} that $V$ is indeed an isomorphism given by the identity on $\Z$ and Poincar\'e duality on $H^2(N;\Z/2)$. Since the surgery obstruction map factors as
		\[\cN(N)\cong \Z\oplus H^2(N;\Z/2)\stackrel{\id\oplus \text{\normalfont PD}}{\cong}\Z\oplus H_2(N;\Z/2)\xrightarrow{I_0+\kappa^s_2}L^s_4(\Z\pi)\to L_4^h(\Z\pi),\]
		it also factors through $V$ as claimed.
	\end{proof}

    \subsection{The transfer in $L$-theory}
    \label{subsec:transfer_L}

    The following will be used in \cref{s:stablerigidity}.
    For a subgroup $H\leq G$ of finite index, let $\tr\colon H_2(G;\Z/2)\to H_2(H;\Z/2)$ be the transfer map.
    
    \begin{proposition}
    \label{prop:transfer-l-theory}
    Fix a choice of decoration $X \in \{h,s\}$.
        Let $H\leq G$ be a subgroup of finite index. Then there exists a transfer map
        \[p^*\colon L_4^X(\Z G)\to L_4^X(\Z H)\]
        such that for all $x\in H_2(G;\Z/2)$ we have $p^*(\kappa_2^X(x))=\kappa_2^X(\tr(x))$.
    \end{proposition}
    
    \begin{proof}
        For each $X \in \{h,s\}$, the transfer map $p^*$ is constructed in \cite{Lueck-Ranicki-Transfer} and we will use the following property. Let $f\colon M\to N$ be a degree one normal map with $\pi_1(N)=G$. Let $g\colon \ol{M}\to \ol{N}$ be the covering of $f$ with $\pi_1(\wt N)=H$. By \cite[Theorem~6.2]{Lueck-Ranicki-Transfer} (see also \cite[Equation 15.470]{luecksurgery}), we have $\theta^X(g)=p^*(\theta^X(f))$ for $X=h$. This also holds for $X=s$ by \cite[Remark 9.7]{Lueck-Ranicki-Transfer}.

        For any manifold $N$ with a $\pi_1$-isomorphism $c\colon N\to B G$ one can show from the Serre spectral sequence for $\wt{M}\to M\to BG$ that the map $H_2(N;\Z/2)\to H_2(G;\Z/2)$ is surjective. Then for any $x\in H_2(G;\Z/2)$ there exists a degree one normal map $f\colon M\to N$ such that $\kappa_2^X(x)=\theta^X(f)$, i.e.\ such that $x=c_*(\kervaire(f)\frown[N])$. Then
        \[p^*\kappa_2^X(c_*(\kervaire(f)\frown[N]))=p^*\theta^X(f)=\theta^X(g)=\kappa_2^X(c'_*(\kervaire(g)\frown[\ol{N}])),\]
    where $c'\colon \ol{N}\to BH$ is the identity on fundamental groups. It remains to show that 
    \[ \tr(c_*(\kervaire(f)\frown[N]))=c'_*(\kervaire(g)\frown[\ol N]).\] 
    By naturality of the transfer map, we have to show $\tr(\kervaire(f)\frown[N])=\kervaire(g)\frown[\ol N]$. For finite coverings of manifolds, the transfer map agrees with the Umkehr map and hence we have to show that $p^*(\kervaire(f))=\kervaire(g)$, where $p\colon \ol N\to N$. This follows from the definition of $\kervaire$.
    \end{proof}

\section{Proof of \cref{thmx:main-table}}\label{s:proofs}
 
	Using the statements from the previous section, in particular \cref{thm:davis-general,cor:davis} we can now prove our main theorem. For this we use the following proposition.

	\begin{proposition} 
    Let $\pi$ be a group and $w\in H^2(\pi;\Z/2)$, let $p\colon H_2(\pi;\Z/2)\to E_{2,2}^\infty(\xi(\pi,w))$ be the projection and let $M,N$ be two 
    manifolds with normal $1$-type $\xi(\pi,w)$. 
    Then
		\begin{clist}{(i)}
			\item $N$ is stably homotopy equivalent to $M$ if and only if $N$ and $M$ have the same signature and there exists a choice of a normal $1$-smoothing such that $\ter(N\#\overline{M})\in p(\ker(\kappa_2^h))$.
			\item $N$ is stably simple homotopy equivalent to $M$ if and only if $N$ and $M$ have the same signature and there exists a choice of a normal $1$-smoothing such that $\ter(N\#\overline{M})\in p(\ker(\kappa_2^s))$.
		\end{clist}
	\end{proposition}
 
	\begin{proof}
		(i) First assume that $N$ is stably homotopy equivalent to $M$. Since the signature is a homotopy invariant and unchanged by stabilisation, the signatures of $M$ and $N$ agree. Since $\ter$ is unchanged by stabilisation, we can assume $N$ and $M$ are homotopy equivalent. Let $f\colon N\xrightarrow{\simeq}M$ be given. Then $\eta(f)\in \cN(M)$ is in the kernel of $\theta^h\colon \cN(M)\to L_4^h(\Z\pi)$. 
  Applying $c_*V$ sends $\eta(f)$ to $\ter(N\#\overline{M})\in F_{2,2}(\xi(\pi,w))$, where $c\colon M\to B\pi$ is the classifying map and for the definition of $V$ see \cref{cor:davis}. Hence $\ter(N\#\overline{M}) \in p(\ker(\kappa_2^h))$, as claimed.
		
		Now assume that there exists $x\in \ker(\kappa_2^h)$ such that $p(x) = \ter(N\#\overline{M})$. Since the surgery sequence in dimension four is exact stably, there exists a homotopy equivalence $g\colon N'\xrightarrow{\simeq} M\#k(S^2\times S^2)$ with $\eta(g)=(x,0)\in \cN(M\#k(S^2\times S^2))\cong \cN(M)\oplus H^2(\#k(S^2\times S^2);\Z/2)$. Then $\ter(N'\#\overline{M})=c_*V(\eta(g))=\ter(N\#\overline{M})$ and hence $\ter(N'\#\overline{N})=0$. Since $g$ is a homotopy equivalence, $N'$ has the same signature as $M$ and thus also the same signature as $N$. It follows that $N$ and $N'$ are stably homeomorphic. Hence $N$ is stably homotopy equivalent to $M$ as claimed.
		
		(ii) The proof is the same as for (i), replacing $\theta^h$ by $\theta^s$ and $\kappa_2^h$ by $\kappa_2^s$.
	\end{proof}
 
    We first show the following analogue of \cref{thmx:main-table} in the topological category. Let $\homeo$ denote homeomorphism.
    
    \begin{theorem}
    \label{thm:main-top}
Let $M$, $M'$ be closed, oriented, almost spin $4$-manifolds with normal $1$-type $\xi=\xi(\pi,w)$. 
 Then, for each subgroup $A \le \Omega(\xi)$ listed below, there is in one-to-one correspondence between $\xi$-bordism mod $A$ and a geometric equivalence relation as follows.
 \FloatBarrier
\bgroup
 \def\x{3.5}
 \def\y{0}
 \vspace{-\x mm}
\def\arraystretch{1.2}
\begin{figure}[h] 
\begin{center}
\setlength{\tabcolsep}{4pt}
\begin{tabular}{rlcp{11cm}}
{\normalfont(i)}  & \, $0$ & $\leftrightarrow$& Stable homeomorphism\\
{\normalfont(ii)}  & \, $F_{0,4}$ & $\leftrightarrow$& There exist simply connected spin $4$-manifolds $L,L'$ such that \hspace{10mm} $M \# L \homeo M' \# L'$ \\
{\normalfont(iii)} & \, $[\ker(\kappa_2^s)]$ & $\leftrightarrow$& Simple homotopy equivalence up to stabilisations by $S^2 \times S^2$ \\
{\normalfont(iv)} & \, $[\ker(\kappa_2^h)]$ & $\leftrightarrow$& Homotopy equivalence up to stabilisations by $S^2 \times S^2$ \\
{\normalfont(v)} & \, $F_{2,2}$ & $\leftrightarrow$& There exist $k,k'\ge 0$ and a $2$-connected degree one normal map $f\colon M\#k(S^2\times S^2)\to M' \# k' (S^2 \times S^2)$\\
{\normalfont(vi)} & \, $F_{3,1}$ & $\leftrightarrow$& There exist simply connected $4$-manifolds $K,K'$ such that \hspace{10mm} $M\# K \homeo M'\# K'$. \\
\end{tabular}
\end{center}
\end{figure}
\vspace{-\x mm}
\vspace{-1mm}
\egroup  
\FloatBarrier
\end{theorem} 
    
\begin{proof}
We begin by proving (i), (ii) and (vi). First note that (i) is \cite[Theorem~C]{surgeryandduality} and (vi) is a reformulation of \cite[Theorem~2.1]{KPT21}. Next observe that (ii) follows from (i). Indeed, two classes in $\Omega_4(\xi)$ differ by an element of $F_{0,4}=H_0(B\pi;\Omega_4^{\topspin})$ if and only if there exists a simply connected spin $4$-manifold representing the difference. So $M$ and $M'$ admit normal $1$-smoothings such that the difference lies in $F_{0,4}$ if and only if there exists a simply connected spin $4$-manifold $L''$ such that $M\#L''$ and $M'$ are stably diffeomorphic. Hence there exist $k,l\in\N$ such that $M\# L\cong M'\#L'$ for $L:=L''\#k(S^2\times S^2)$ and $L':=l(S^2\times S^2)$. Conversely, if $M\#L\cong M'\#L'$, then there are normal $1$-smoothings of $M$ and $M'$ such that their difference is the same as $[L]-[L']\in H_0(B\pi;\Omega_4^{\topspin})=F_{0,4}$.

 We now prove (v). Let $\xi\colon B\to \BSTop$ be the normal $1$-type of $M$ and $M'$. Assume that there exist normal $1$-smoothings $f\colon M\to B$ and $f'\colon M'\to B$ such that their difference lies in $F_{2,2}(\xi)$. The map $F_{2,2}(\xi(M',w_2(M')))\xrightarrow{c_*} F_{2,2}(\xi)$ is surjective. Hence there exists a degree one normal map $g\colon N\to M'$ with $c_*V(g)=[M,f]-[M',f']$ by \cref{cor:davis}. Surgering $N$ if necessary, we can assume that $g$ is $2$-connected. By construction, $c_*V$ is given by sending $g$ to $[N,\wt\nu_N]-[M',f']$ for some normal $1$-smoothing $\wt\nu_N$ of $N$. Therefore $[M,f] = [N,\wt \nu_N]$, and so it follows that $M$ and $N$ are stably homeomorphic. Hence, as claimed, there exists a $2$-connected degree one normal map		
\[M\#k(S^2\times S^2)\cong N\# k'(S^2\times S^2)\to M'\# k'(S^2\times S^2).\]
		
Conversely, assume that there exists a $2$-connected degree one normal map $M\#k(S^2\times S^2)\to M'\#k'(S^2\times S^2)$. We can compose this with the canonical map $M'\#k'(S^2\times S^2)\to M'$ to get a degree one normal map $g\colon M\#k(S^2\times S^2)\to M'$. Then by \cref{thm:davis-general}, there is a $1$-smoothing $\wh g\colon M\#k(S^2\times S^2)\to B_{\xi(M',w_2(M'))}$ such that       
\[[M\#k(S^2\times S^2),\wh g]-[M',\wt{\nu M'}]\in F_{2,2}(\xi(M',w_2(M')).\] 
Composing with the map $B_{\xi(M',w_2(M'))}\to B$ we obtain normal $1$-smoothings of $M'$ and $M$ such that the difference lies in $F_{2,2}(\xi)$ as claimed.
  
Next we prove (iv). Let $f\colon M'\#k(S^2\times S^2)\to M\#l(S^2\times S^2)$ be a homotopy equivalence.
Then $\eta(f)=(0,y)\in \Z\oplus H_2(M\#l(S^2\times S^2);\Z/2)$ with $c_*y\in \ker(\kappa_2^h)$.
    Hence $c_*V(\eta(f))=[M'\# \overline{M}]=[M']-[M]\in [\ker(\kappa_2^h)]$ for some choice of $\xi$ structure on $M'$ and $M$.
    
		Now assume that there are $\xi$-structures on $M$ and $M'$ such that $[M'\#\overline{M}]$ lies in the image of $\ker(\kappa^h_2)$ in $F_{2,2}(\xi)$. Then there exists a degree one normal map $f\colon N\to M$ with $\theta^h(f)=0$ and $c_*V(f)=[M'\#\overline{M}]$. Since $\theta^h(f)=0$, the exactness of the stable surgery exact sequence \cite[Theorem~4]{kirby-taylor} implies the existence of a homotopy equivalence $f'\colon N'\xrightarrow{\simeq} M\#k(S^2\times S^2)$ such that $\eta(f')$ maps to $f\in \cN(M)$ under the collapse map $M\#k(S^2\times S^2)\to M$. Thus also $[N'\#\overline{M}]=c_*V(f')=[M'\#\overline{M}]$ and thus $[N']=[M']\in \Omega_4(\xi)$ and $N'$ and $M'$ are stably diffeomorphic by (i). Hence $M'$ is stably homotopy equivalent to $M$ as claimed.
		
		The proof of (iii) is same, replacing $\kappa_2^h$ by $\kappa_2^s$ and $\theta^h$ by $\theta^s$.
	\end{proof}

\begin{remark} \label{rem:eq-rel-direct} 
     A priori it is not obvious that the existence of a $2$-connected degree one normal map $f\colon M\#k(S^2\times S^2)\to M' \# k' (S^2 \times S^2)$ is an equivalence relation. We briefly sketch a direct argument, i.e.\ we show that if $f\colon N\to M$ is a degree one normal map that is $2$-connected, then there also exists an $2$-connected degree one normal map $M\#k(S^2\times S^2)\to N\#k'(S^2\times S^2)$ for some $k,k'\in \N$. 

     We can assume that the restriction of $f$ to $N^{(1)}$ is a homeomorphism to $M^{(1)}$ \cite[Lemma 8.3]{Ferry_Pedersen_1995}, after possibly stabilising $M$. We can then consider the $2$-connected degree one normal map given by the connected sum over the $1$-skeleton of $f$, $\id_{\overline N}$ and $\id_{\overline M}$. Note that $N\#_1\overline{N}\#_1\overline{M}$ is stably homeomorphic to $\overline{M}$ while $M\#_1\overline{N}\#_1\overline{M}$ is stably homeomorphic to $\overline{N}$. Thus after changing the orientation and stablising if necessary, this yields a $2$-connected degree one normal map $M\#k(S^2\times S^2)\to N\#k'(S^2\times S^2)$ for some $k,k'\in \N$ as needed.
 \end{remark}

    To deduce \cref{thmx:main-table} from \cref{thm:main-top}, we need the following results.
    
    \begin{lemma}
\label{lem:ter-ks}
Let $M$ be an oriented almost spin manifold with $\pri M = \sec M = 0$ and let $w \in
H^2(\pi_1M; \Z/2)$ be such that $c_M^*w = w_2(M)$. Then $w\colon H_2(\pi_1(M);\Z/2)\to \Z/2$ factors through $E^\infty_{2,2}(\xi(\pi,w))$, $\sigma(M)$ is divisible by $8$ and
$\ks(M)\equiv \sigma(M)/8+ w\frown\ter M \mod 2$.
\end{lemma}

\begin{proof}
    Let $N$ be a smooth 4-manifold with normal $1$-type $\xi:=\xi(\pi,w)$ that is nullbordant, i.e.\ represents $0\in\Omega_4(\xi)$. Let $c\colon N\to B\pi$ be a map inducing the chosen identification of fundamental groups such that $c^*w=w_2(N)$. By \cref{thm:main-top}, there exists a degree one normal map $f\colon M\#k(S^2\times S^2)\to N$. It follows that $\sigma(M)=\sigma(M)-\sigma(N)$ is divisible by 8.

    Let $[f]=(\sigma(M)/8,x)\in \Z\oplus H_2(N;\Z/2)\cong \cN(N)$. By \cref{thm:davis-general}  there are $1$-smoothings of $M$,$N$ in $\xi(N,w_2(N))$ such that $[M,\wt{\nu}_M]-[N,\wt{\nu}_N]\in F_{2,2}(\xi(N,w_2(N)))$. The map $c\colon N\to B\pi$ induces a map of structures $\xi(N,w_2)\to \xi(\pi,w)$ and we get that $x$ is a lift of $\ter(M)$ along $H_2(N;\Z/2)\xrightarrow{c_*} H_2(\pi;\Z/2)\to E^\infty_{2,2}(\xi)$. By \cite{KirbySiebenmann}*{p.~329}, 
        \[\ks(M)=\ks(M)-\ks(N)\equiv \sigma(M)/8+w_2(N)\frown x=\sigma(M)/8+w\frown c_*x \mod 2.\]
    Here $c_*x$ is a lift of $\ter(M)$ to $H_2(\pi;\Z/2)$.
    For any other lift $y$ of $\ter(M)$ to $H_2(\pi;\Z/2)$, there exists again a degree 1 normal map $f'\colon M'\to N$ such that $[f']=(\sigma(M)/8,x')$ with $x'$ a lift of $y$ along $c_*$. In particular, we choose $M'$ such that $\sigma(M')=\sigma(M)$. It follows from \cref{thm:davis-general}, that $M$ and $M'$ represent the same class in $\Omega_4(\xi)$ and thus are stably homeomorphic. In particular, $\ks(M)=\ks(M')$. Again applying the formula from \cite{KirbySiebenmann}*{p.~329} we have
    \[\ks(M)=\ks(M')=\ks(M')-\ks(N)=\sigma(M')/8+w_2(N)\frown x'=\sigma(M)/8+w\frown y \mod 2.\]
    Comparing the two computations of $\ks(M)$ it follows that $w\frown y=w\frown c_*x$ and thus $w$ factors through $E_{2,2}^\infty(\xi)$ as claimed and we obtain the proposed formula for $\ks(M)$.
\end{proof}

    \begin{corollary}
    \label{cor:lifting-ker-kappa-smooth}
    Let $M$ be a smooth, orientable, almost spin manifold with normal $1$-type $\xi^{\ttop}=\xi(\pi,w)$ over $\BSTop$. Let $\xi^{\diff}$ be the normal $1$-type of $M$ over $\BSO$ defined by $(\pi,w)$. There is an inclusion $F_{2,2}(\xi^{\diff})\to F_{2,2}(\xi^{\ttop})$ and an element $x$ of $\ker(\kappa_2^h)$ or $\ker(\kappa_2^s)$ is in the image of $F_{2,2}(\xi^{\Diff})$ if and only if $w\frown x=0$. In particular, we can consider $\xi^{\Diff}$-bordism modulo $[\ker(\kappa_2^h)\cap\ker(w\frown-)]$ or $[\ker(\kappa_2^s)\cap\ker(w\frown-)]$.
    \end{corollary}
    
    \begin{proof}
        Since we have $\Omega_*^{\spin}\cong \Omega_*^{\topspin}$  for $*\leq 3$ and the inclusion  $8\Z\cong\Omega_4^{\spin}\leq \Omega_4^{\topspin}\cong 16\Z$, the map $F_{2,2}(\xi^{\diff})\to F_{2,2}(\xi^{\ttop})$ coming from naturality of the James spectral sequence is an inclusion. The image are those bordism classes with trivial Kirby--Siebenmann invariant. By \cref{lem:ter-ks}, an element $x$ of $\ker(\kappa_2^h)$ or $\kappa_2^s$ is mapped to the image of $F_{2,2}(\xi^{\diff})$ if and only if $x\frown w=0$.
    \end{proof}

    \begin{proof}[Proof of \cref{thmx:main-table}]
        (i), (ii) and (vi) are proven as in \cref{thm:main-top}.

        Let $\xi^{\diff}$ be the normal $1$-type of $M$ over $\BSO$ and let $\xi^{\ttop}$ be the normal $1$-type of $M$ over $\BSTop$. Since $\Omega_*^{\spin}=\Omega_*^{\topspin}$ for $*\leq 3$, we have $E_{4,0}^\infty(\xi^\diff)=E_{4,0}^\infty(\xi^{\ttop})$, $E_{3,1}^\infty(\xi^\diff)=E_{3,1}^\infty(\xi^{\ttop})$ and $E_{2,2}^\infty(\xi^\diff)=E_{2,2}^\infty(\xi^{\ttop})$ by naturality of the James spectral sequence. In particular, $M$ and $M'$ are $\xi^\diff$-bordant mod $F_{2,2}(\xi^{\diff})$ if and only if they are $\xi^{\ttop}$-bordant mod $F_{2,2}(\xi^{\ttop})$. Hence (v) follows from \cref{thm:main-top} (v).

        (iii) and (iv) follow from \cref{thm:main-top} using \cref{cor:lifting-ker-kappa-smooth}.
     \end{proof}
 
\section{Stable rigidity} \label{s:stablerigidity}

Recall from \cref{def:stable-rigid} that a finitely presented group $\pi$ satisfies \emph{stable rigidity} if any two closed, oriented, smooth, homotopy equivalent $4$-manifolds with fundamental group $\pi$ are stably diffeomorphic (or equivalently stably homeomorphic).
This terminology is inspired by the Borel conjecture which states that aspherical closed manifolds (in any dimension) are determined up to homeomorphism by their homotopy type, i.e. by the fundamental group. 

We will start by studying the analogous notion in the topological category, which we refer to as \textit{strong stable rigidity} (\cref{def:strong-stable-rigidity}), before establishing \cref{prop:stable-he-vs-homeo} and \cref{cor:stable-rigidity,cor:smallf22} which are the basis for our analyses of whether given groups satisfy stable rigidity.

For the remainder of this section, we will then determine when stable rigidity holds for a range of groups $\pi$. In \cref{ss:stable-rigidity-sylow}, we relate stable rigidity for $\pi$ to that of its odd index subgroups. In \cref{ss:abelian,ss:quaternion,ss:Dih_SemiDih,ss:modular-maximal}, we then determine stable rigidity for groups which are abelian, quaternion, dihedral, semi-dihedral, or modular maximal-cyclic.

\subsection{General results for (strong) stable rigidity}
\label{ss:stable-rigidity-general}

We begin by making the following key definition.
This can be viewed as the analogue of stable rigidity in the topological category.

	\begin{definition}
 \label{def:strong-stable-rigidity}
		A group $\pi$ satisfies \emph{strong stable rigidity} if any two closed, oriented, almost spin, homotopy equivalent manifolds $M$ and $M'$ with fundamental group $\pi$ are stably homeomorphic.
	\end{definition}

\begin{remark}
\label{rem:ssr}
\begin{clist}{(a)}
    \item\label{it:ssr-a} For any finitely presented group $\pi$ there exist closed, oriented, homotopy equivalent manifolds with fundamental group $\pi$ that are not stably homeomorphic. A pair of such manifolds can be constructed by starting with any closed, oriented manifold $M$ with fundamental group $\pi$ and considering $M\#\CP^2$ and $M\#{*\CP^2}$, where ${*\CP^2}$ is the Chern manifold constructed by Freedman in \cite[p.~370]{Freedman82}. For this reason \cref{def:strong-stable-rigidity} is restricted to almost spin manifolds.
    \item\label{it:ssr-b} In contrast, two closed, oriented, smooth, totally non-spin manifolds $M$ and $M'$ are stably diffeomorphic if and only if they have the same signature and the images of their fundamental classes in $H_4(\pi;\Z)/\Out(\pi)$ agree. This follows from \cite[Theorem~C]{surgeryandduality}, see also \cite[Lemma~2.1]{KPT}. In particular, they are stably diffeomorphic if they are homotopy equivalent.  
    \item A group $\pi$ that satisfies strong stable rigidity also satisfies stable rigidity. This can be seen as follows. Let $M$, $M'$ be a pair of closed, oriented, smooth, homotopy equivalent manifolds with fundamental group $\pi$. If $M$ and $M'$ are totally non-spin, then they are stably diffeomorphic by \eqref{it:ssr-b}. If $M$ and $M'$ are almost spin, they are stably diffeomorphic if $\pi$ satisfies stable rigidity.
    \item We could define (strong) \emph{simple} stable rigidity by changing homotopy equivalence to \emph{simple} homotopy equivalence in \cref{def:stable-rigid,def:strong-stable-rigidity}. Then the following results about (strong) stable rigidity also hold for (strong) simple stable rigidity if $\ker(\kappa_2^h)$ is substituted by $\ker(\kappa_2^s)$.
\end{clist}
\end{remark}

	As mentioned in the introduction, \cref{thmx:main-table} recovers the result of Davis~\cite{Davis-05}*{Theorem~1.4} that a group satisfies strong stable rigidity if $\kappa_2^h$ is injective for $\pi$ as a corollary. It turns out that the converse is also true. While our proof of \cref{thmx:main-table} uses the methods from \cite{Davis-05}, there is a shorter direct proof which we will give here.

	\begin{proposition}
		\label{prop:davis-intro}
		A group $\pi$ satisfies strong stable rigidity if and only if $\kappa_2^h$ is injective for $\pi$.
	\end{proposition}
 
\begin{proof}[Proof (alternative).] 
By definition of strong stable rigidity, we only consider manifolds that are almost spin.
Consider the composition $p\colon\cN(M)\cong H^4(M;\Z) \oplus H^2(M;\Z/2)\to H^2(M;\Z/2)$.
If $\kappa_2^h$ is injective, we claim that for every homotopy equivalence $f\colon M'\to M$ the Poincare dual $PD(p(\eta(f)))$ of the normal invariant $\eta(f)$ of $f$ is represented by an immersed sphere $\alpha\colon S^2\to M$. This is true since $PD(p(\eta(f)))$ maps to zero in $H_2(\pi;\Z/2)$ because $\kappa_2^h$ is injective. Hence its image in $H_2(M;\Z/2)$ comes from the universal cover and can thus be represented by an immersed sphere as claimed. Since $M$ is almost spin, $w_2(M)$ vanishes on this sphere. Using Novikov pinching \cite[Theorem~5.1]{cochran-habegger}, see also \cite[Lemma~3.3]{KL22}, we can realize $-\eta(f)$ by a self-homotopy equivalence $h\colon M\to M$ that is the identity on $\pi_1$. Then $p(\eta(h\circ f))=p(\eta(h))+p(\eta(f))=0$.
	  
    By the stable surgery exact sequence \cite{kirby-taylor}, $M$ and $M'$ stably differ by the action of $L_5(\Z\pi)$. In particular, they are stably homeomorphic.

        Now assume that $\kappa_2^h$ is not injective. Let $x\in H_2(\pi;\Z/2)$ be in the kernel of $\kappa_2^h$ and let $w\in H^2(\pi;\Z/2)$ be such that $w\frown x=1$. Let $M$ be an almost spin 4-manifold with fundamental group $\pi$ and $w_2(M)=c^*w$ where $c\colon M\to B\pi$ is the classifying map. Let $x'\in H_2(M;\Z/2)$ be a preimage of $x$ under $c_*$. Again using the stable surgery exact sequence \cite{kirby-taylor}, there exists a homotopy equivalence $f\colon M'\to M\#k(S^2\times S^2)$ with $\eta(f)=(0,x')\in \Z\oplus H_2(M;\Z/2)\cong \cN(M)$ for some $k\in\N$. By \cite{KirbySiebenmann}*{p.~329},         \[\ks(M)-\ks(M')=w_2(M)\frown x'=w\frown x=1\]
        Hence $M$ and $M'$ have different Kirby--Siebenmann invariants and thus are not stably homeomorphic. It follows that $\pi$ does not satisfy strong stable rigidity.
\end{proof}

	The Farrell--Jones conjecture predicts that $\kappa_2^h$ is injective for every torsion free group as observed by Davis~\cite{Davis-05}, see also \cite[Lemma~3.3]{hambleton_stability_2022}. While the Farrell--Jones conjecture is not known in general, it holds for many classes of groups, for example hyperbolic groups \cite{FJC-hyperbolic}, CAT(0) groups \cite{BL,Wegner-Cat(0)}, $3$-manifold groups \cite{Bartels-Farrell-Luck}, and solvable groups \cite{Wegner}.

See \cite[Proposition~1]{teichner-star} for an example of homotopy equivalent, orientable, almost spin 4-manifolds with fundamental group $\Z/2$ and different Kirby--Siebenmann invariant. In particular, $\Z/2$ does not satisfy strong stable rigidity. This can be viewed as a new proof that $\kappa_2^h$ is not injective for $\Z/2$.

	\begin{proposition}
		\label{prop:stable-he-vs-homeo}
		Let a group $\pi$ and $w\in H^2(\pi;\Z/2)$ be given. 
		There exists a pair of stably homotopy equivalent but not stably homeomorphic manifolds with normal $1$-type $\xi:=\xi(\pi,w)$ if and only if there exists an element $x\in \ker(\kappa_2^h)\leq H_2(\pi;\Z/2)$ that maps to a non-trivial element in $E^\infty_{2,2}$. Moreover, there exists such a pair with the same Kirby--Siebenmann invariant if and only if we can pick $x$ such that $w\frown x=0$.	
	\end{proposition}

	\begin{proof}
    Let $x\in \ker\kappa_2^h$ and let $\overline{x}$ be the image of $x$ in $E^\infty_{2,2}$. Then there exists a 4-manifold $M$ with normal $1$-type $\xi$ and a normal $1$-smoothing $f$ such that $(M,f)=(0,\overline{x})\in F_{2,2}$ and a manifold $N$ with normal $1$-type $\xi$ and $\xi$-nullbordant normal $1$-smoothing $g$. By \cref{thm:main-top}, $M$ and $N$ are stably homotopy equivalent since $([M,f]-[N,g])=(0,\overline{x})\in F_{2,2}$ is in the image of $\ker(\kappa_2^h)$ by assumption. Since $0\neq ([M,f]-[N,g])\in F_{2,2}$, $M$ and $N$ are not stably homeomorphic. 
		
		Conversely, assume that $M$ and $M'$ are stably homotopy equivalent but not stably homeomorphic. Then there exists normal $1$-smoothings $f$ and $f'$ such that $[M,f]-[M,f']\in F_{2,2}$ is non-trivial but in the image of $\ker(\kappa_2^h)$ by \cref{thm:main-top}.
		
		The second statement follows from the fact that $M$ and $M'$ with $[M,f]-[M',f']=(0,[x])\in F_{2,2}$ have the same Kirby--Siebenmann invariant if and only if $w\frown x=0$ by \cref{lem:ter-ks}.
	\end{proof}

 \begin{corollary}
 \label{cor:stable-rigidity}
Let $\pi$ be a group. Then $\pi$ does not satisfy (strong) stable rigidity if and only if there is $w\in H^2(\pi;\Z/2)$ and $x\in \ker(\kappa_2^h)$ such that x maps to a nontrivial element in $E_{2,2}^\infty(\xi(\pi,w))$ (and in the case of strong rigidity we can find such $x$ such that $w\frown x=0$).
 \end{corollary}
 
 \begin{proof}
     Most of the corollary follows directly from  \cref{prop:stable-he-vs-homeo}. We need to show that if there is a $w\in H^2(\pi;\Z/2)$ and $x\in \ker(\kappa_2^h)$ such that $x$ maps to a nontrivial element in $E_{2,2}^\infty(\xi(\pi,w))$ and $w\frown x=0$ then we can find \emph{smooth} manifolds representing $x$. \cref{prop:stable-he-vs-homeo} gives us a pair of stably homotopy equivalent but not stably homeomorphic manifolds $M,N$ with the same Kirby--Siebenmann invariant. If $ks(M)=ks(N)=0$, then $M$ and $N$ are stably smoothable. If $ks(M)=ks(N)=1$, then $M\#E_8$ and $N\#E_8$ are stably smoothable and $[M\#E_8]-[N\#E_8]=[M]-[N]=x$.  
 \end{proof}
 
 \begin{corollary}
 \label{cor:smallf22}
     Let $\pi$ be a group such that $E^\infty_{2,2}(\xi(\pi,0))=0$ and $E^\infty_{2,2}(\xi(\pi,w))\cong \Z/2$ for $0\neq w\in H^2(\pi;\Z/2)$. Then $\pi$ satisfies stable rigidity.
\end{corollary}

 \begin{proof}
     Let $x\in \ker(\kappa_2^h)$. Assume that there exists $w\in H^2(\pi;\Z/2)$ such that 
     $0\neq[x]\in E^\infty_{2,2}(\xi(\pi,w))$. Then by assumption $w\neq 0$. By \cref{lem:ter-ks}, $w$ factors through $E^\infty_{2,2}(\xi(\pi,w))\cong \Z/2$. Hence $w\frown x\neq 0$ and it follows from \cref{cor:stable-rigidity} that $\pi$ satisfies stable rigidity. 
 \end{proof}

\subsection{Stable rigidity for subgroups of odd index}
\label{ss:stable-rigidity-sylow}
 
We now consider some results relating stable rigidity of $\pi$ to its $2$-Sylow subgroup.

\begin{lemma}
\label{lem:surj-ker-kappa}
    Let $f\colon G\to \pi$ be a group homomorphism inducing a surjection on $\ker(\kappa_2^h)$. If $G$ satisfies (strong) stable rigidity, so does $\pi$.
\end{lemma}

\begin{proof}
    If $G$ satisfies strong stable rigidity, then $\kappa^h_2$ is injective for $G$ and hence also for $\pi$ by assumption. Thus $\pi$ satisfies strong stable rigidity.

    Let $x\in H_2(\pi;\Z/2)$ be in $\ker(\kappa_2^h)$ and $w\in H^2(\pi;\Z/2)$ such that $w\frown x=0$. Then there exists a preimage $x'\in H_2(G;\Z/2)$ that also is in the kernel of $\kappa^h_2$. Since $G$ satisfies stable rigidity and $f^*w\frown x'=w\frown x=0$, $0=[x']\in E_{2,2}^\infty(\xi(G,f^*w))$. By naturality also $0=[x]\in E_{2,2}^\infty(\xi(\pi,w))$. It follows from \cref{prop:stable-he-vs-homeo} that $\pi$ satisfies stable rigidity.
\end{proof}

\begin{proposition}
\label{prop:oddindex}
    Let $G\leq \pi$ be a subgroup of odd index. 
    If $G$ satisfies (strong) stable rigidity, so does $\pi$.
\end{proposition}

\begin{proof}
Let $i\colon G\to\pi$ be the inclusion. Let $x\in H_2(\pi;\Z/2)$ be in the kernel of $\kappa^h_2$. Let $x':=\tr(x)\in H_2(G;\Z/2)$, where $\tr$ is the transfer map. Since $G$ has odd index, $\tr\circ i_*$ is the identity and thus $x'$ is a preimage of $x$. By \cref{prop:transfer-l-theory}, $\kappa^h_2(x')=p^*(\kappa^h_2(x))=0$. Hence $i$ induces a surjection on the kernel of $\kappa^h_2$. The proposition now follows from \cref{lem:surj-ker-kappa}.
\end{proof}

\begin{proposition}\label{prop:odd-order-normal-sg}
	Let $1\to P\to \pi\to G\to 1$ be a short exact sequence with $P$ a finite group of odd order. 
	If $G$ satisfies (strong) stable rigidity, then $\pi$ satisfies (strong) stable rigidity. 
	
	Moreover, if the above sequence splits, then the converse holds. That is, $G$ satisfies (strong) stable rigidity if and only if $\pi$ satisfies (strong) stable rigidity. 
\end{proposition}

\begin{proof}
If $\pi$ does not satisfy stable rigidity, then \cref{prop:stable-he-vs-homeo} implies that there exists $w\in H^2(\pi;\Z/2)$ and $x\in \ker(\kappa_2^h) \leq H_2(\pi;\Z/2)$ such that the image of $x$ is non-trivial in $F_{2,2}\leq \Omega_4(\xi(\pi,w))$.
Similarly, if $\pi$ does not satisfy strong stable rigidity, then there exists $w\in H^2(\pi;\Z/2)$ and $x\in \ker(\kappa_2^h) \cap \ker(-\frown w) \leq H_2(\pi;\Z/2)$ with the corresponding properties.
 
	Since $\wt{H}_*(P;\Z/2)=0$, the Lyndon-Hochschild-Serre spectral sequence gives us the isomorphism of the reduced homology $p_*\colon\wt{H}_*(\pi;\Z/2)\to \wt{H}_*(G,\Z/2)$ induced by the projection $p\colon \pi\to G$. Dually, $p^*$ is an isomorphism and there exists $w'\in H^2(G;\Z/2)$ with $p^*w'=w$.
	
	We now compare the spectral sequences for $\Omega_4(\xi(\pi,w))$ and $\Omega_4(\xi(G,w'))$.
	Recall that the map
$d_2\colon H_4(\pi;\Z/2)\to H_2(\pi;\Z/2)$ 
 is given by $\Sq_2+-\frown w$. We have
\[\Sq_2(p_*y)+p_*y\frown w'=p_*(\Sq_2y+y\frown p^*w')=p_*(\Sq_2y+y\frown w).\]
	Since $p_*\colon H_4(\pi;\Z/2) \to H_4(G;\Z/2)$ is an isomorphism, we have an isomorphism 
 \[ H_2(\pi;\Z/2)/\im d_2\xrightarrow{\cong}H_2(G;\Z/2)/\im d_2.\]
	Consider the commutative diagram
	\[\begin{tikzcd}
		\ker(d_2\colon H_5(\pi;\Z)\to H_3(\pi;\Z/2))\ar[d,"d_3"]\ar[r]&\ker(d_2\colon H_5(G;\Z)\to H_3(G;\Z/2))\ar[d,"d_3"]\\
		H_2(\pi;\Z/2)/\im d_2\ar[r,"\cong"]&H_2(G;\Z/2)/\im d_2
	\end{tikzcd}\]
	We get $p_*\im(d_3)\subset \im(d_3)\subset H_2(G;\Z/2)/\im d_2$. This shows that $p$ induces an injection 
	\[H_2(\pi;\Z/2)/d_2d_3\xrightarrow{p_*}H_2(G;\Z/2)/d_2d_3.\]
	
	It follows the image of $p_*(x)\in \ker(\kappa_2^h(G))$ in $H_2(G;\Z/2)/d_2d_3$ is non-trivial. If $x\frown w=0$, then $p_*x\frown w'=p_*(x\frown p^*w')=p_*(x\frown w)=0$. Hence $G$ does not satisfy (strong) stable rigidity by \cref{prop:stable-he-vs-homeo}. This proves the first part of the corollary.
	
	We now assume that the short exact sequence splits. In this case $p_*\colon H_5(\pi;\Z)\to H_5(G;\Z)$ is surjective and it follows from the commutative diagram above, that 
\[
H_2(\pi;\Z/2)/d_2d_3\xrightarrow{p_*}H_2(G;\Z/2)/d_2d_3\] 
 is an isomorphism. Moreover, it follows from the naturality of $\kappa_2^h$ that $p_*$ restrict to an isomorphism on the kernels of $\kappa_2^h$. Thus the second part of the corollary again follows from \cref{prop:stable-he-vs-homeo}.
\end{proof}
	 
 \subsection{Abelian groups}
 \label{ss:abelian}
 
	In this subsection we prove \cref{thm:abelian}. We first show the following lemma which will allow us to reduce to the case of finite abelian groups.
 
    \begin{lemma}
    \label{lem:abelian0}
        Let $A\cong \bigoplus_{i=1}^kC_{n_i}$ be a finitely generated abelian group with $C_{n_i}$ cyclic of order $n_i\in \N\cup\{\infty\}$. Then $\ker(\kappa_2^h)=\ker(\kappa_2^s)$ and
\[x=((x_i)_i,(y_{ij})_{ij})\in \bigoplus_{i=1}^kH_2(C_{n_i};\Z/2)\oplus \bigoplus_{i\neq j}\left(H_1(C_{n_i};\Z/2)\otimes H_1(C_{n_j};\Z/2)\right)\cong H_2(A;\Z/2)\]
is in the kernel of $\kappa_2^h$ if and only if $y_{ij}=0$ for all $i,j$ such that $4$ divides $n_i$ or $n_j$ (which includes the case that one of them is $\infty$). 
    \end{lemma}
    
 \begin{proof}
    First note that $H_1(C_{n_i};\Z/2)\otimes H_1(C_{n_j};\Z/2)=0$ if $n_i$ or $n_j$ is odd. In particular, $y_{ij}=0$ if $n_i$ or $n_j$ is odd.
    Also note that $\ker(\kappa_2^s)\subseteq\ker(\kappa_2^h)$. Hence the statement follows if we show the following two statements for every $x\in H_2(A;\Z/2)$ as in the lemma statement.
    \begin{enumerate}
        \item If $y_{ij}=0$ for all $i,j$ such that $4$ divides $n_i$ or $n_j$, then $x\in \ker(\kappa_2^s)$.
        \item If there exist $i, j$ such that $4$ divides $n_i$ or $n_j$ with $y_{i,j}\neq 0$, then $x\notin \ker(\kappa_2^h)$.
    \end{enumerate}

     By \cite[Theorem 4.1(b)]{TW80}, $\kappa_2^s(G)$ is trivial for $G$ cyclic or $G\cong C_2\times C_2$. 
     (This also follows from \cite[Proposition 7.4]{HMTW88} for $G$ cyclic and, for $G = C_2^2$, from the fact that $H_2(C_2^2;\Z/2)$ is finite and $L_4^s(\Z[C_2^2])$ is torsion-free by \cite[Theorem 3.4.5]{Wa76}.)
     Hence if $y_{ij}=0$ for all $i,j$ such that $4$ divides $n_i$ or $n_j$, $x$ is in the kernel of $\kappa_2^s$. This shows (1). 

    We now show (2). Under the projection $p\colon A\to C_{n_i}\times C_{n_j}$ we have $p_*\kappa_2^h(x)=\kappa_2^h(y_{ij})$. Hence it suffices to show that $\kappa_2^h(y_{ij})\neq 0$ for all $i,j$ such that $4$ divides $n_i$ or $n_j$. By assumption there is a projection $p\colon A\to C_4\times C_2$ such that $p_*(y_{ij})\neq 0\in H_1(\Z/2;\Z/2)\otimes H_1(\Z/4;\Z/2)$. On $p_*(y_{ij})$,  $\kappa_2^h$ is non-trivial as shown by Morgan and Pardon (unpublished), see \cite[Section~4]{Milgram-ranicki} for a proof.
 \end{proof}
 
	\begin{corollary}
		\label{lem:abelian1}
		Let $A$ be a finitely generated abelian group. Let $x\in H_2(A;\Z/2)$ be in the kernel of $\kappa_2^h$. Then there exists a finite subgroup $A'$ such that $x$ is the image of some $x'\in \ker(\kappa_2^h(A'))$. Furthermore $\alpha\frown x'=0$ for every $\alpha\in H^2(A';\Z/2)$ with $\alpha^2=0$.
	\end{corollary}
 
	\begin{proof}
        Since $H_2(C_\infty;\Z/2)=0$, the existence of $x'$ follows from \cref{lem:abelian0}.
 
		We now show that $\alpha\frown x'=0$ for every $\alpha\in H^2(A';\Z/2)$ with $\alpha^2=0$. If $\alpha^2=0$, then $\alpha=\sum_{i,j} \alpha_i\alpha_j$ with $\alpha_i,\alpha_j\in H_1(A';\Z/2)$ and $\alpha_i^2=0$. It suffices to show $(\alpha_i\alpha_j)\frown x'=0$ for all such $\alpha_i,\alpha_j$ or equivalently, that $-\frown x'$ is trivial on the image of every $H^1(C;\Z/2)\otimes H^1(C';\Z/2)$ with $C\cong \Z/4k$ and $C'\cong \Z/2l$ for every map $A'\to C\times C'$. This again follows from \cref{lem:abelian0}.
	\end{proof}
 
	\begin{lemma}
		\label{lem:abelian2}
		Let $A$ be a finite abelian group, Let $f\colon H^2(A;\Z/2)\to \Z/2$ be a map that sends $\alpha$ to zero if $\alpha^2=0$. Let $w\in H^2(A;\Z/2)$ be some element with $f(w)=0$. Then there exists a map $f'\colon H^4(A;\Z/2)\to \Z/2$ such that $f'\circ (\Sq^2+-\smile w)=f$.
	\end{lemma}
 
	\begin{proof}
		As $\Sq^2+-\smile w$ is linear, it suffices to show that the kernel of this map is contained in the kernel of $f$. Let $\alpha\in H^2(A;\Z/2)$ satisfy $0=(\Sq^2(\alpha)+\alpha\smile w)=\alpha(\alpha+w)$. By the structure of $H^*(A;\Z/2)$ this implies either $\alpha^2=0$ and thus $f(\alpha)=0$ or $(\alpha+w)^2=0$. In the latter case $f(\alpha)=f(w)=0$.
	\end{proof}
 
	\begin{proof}[Proof of \cref{thm:abelian}]
		Let $M$ and $M'$ be stably homotopy equivalent with fundamental group $A$. By \cref{rem:ssr}~\eqref{it:ssr-b}, we can assume that $M$ and $M'$ are almost spin. Let $\xi(A,w')$ be their normal $1$-type. Then there are normal $1$-smoothings such that $\ter(M'\sqcup \ol{M})\in F_{2,2}(\xi(A,w'))$ is the image of an element $t'\in \ker(\kappa_2^h)$. If $\ks(M)=\ks(M')$, then $w'\frown t'=0$ by \cref{lem:ter-ks}. By \cref{lem:abelian1}, there is a finite subgroup $A'$ of $A$ such that $t'$ is the image of some $t\in H_2(A';\Z/2)$. We now show that $t$ is in the image of $d_2=(\Sq^2+-\smile w)^*$, where $w$ is the restriction of $w'$ to $A'$. We have $w\frown t=0$. Furthermore, $\alpha\frown t=0$ for every $\alpha\in H^2(A';\Z/2)$ with $\alpha^2=0$. By \cref{lem:abelian2}, $t$ is in the image of $d_2$.
		
		Thus also $t'$ is in the image of $d_2$ and $M$ and $M'$ admit bordant normal $1$-smoothings. Hence they are stably homeomorphic by \cite[Theorem~C]{surgeryandduality}.
	\end{proof}
 
Combining \cref{thm:abelian} with \cref{prop:oddindex} we obtain the following result.

\begin{corollary}
\label{cor:piAbelianStabRig}
    Let $\pi$ be a finite group with abelian $2$-Sylow subgroup, then $\pi$ satisfies stable rigidity.
\end{corollary} 

 \subsection{Quaternion groups}\label{ss:quaternion}

The following determines stable rigidity for finite groups whose Sylow $2$-subgroup is quaternion.

\begin{proposition}
\label{prop:quaternion}
	Let $\pi$ be a finite group such that its $2$-Sylow subgroup $Q$ is quaternionic.
\begin{clist}{(i)}
\item If $|Q|>8$ and $\pi\cong P\rtimes Q$ with $P$ a group of odd order, then $\pi$ does not satisfy stable rigidity.
		\item If $|Q|=8$ and $\pi\cong P\rtimes Q$ with $P$ a group of odd order, then $\pi$ satisfies stable rigidity but not strong stable rigidity.
		\item If $\pi$ is not of the form $P\rtimes Q$, then $\pi$ satisfies strong stable rigidity.
	\end{clist}
\end{proposition}

\begin{remark}
In \cite{teichnerthesis}*{Example~5.2.4}, it is stated that no group with quaternionic $2$-Sylow subgroup satisfies stable rigidity. However the proof relies on \cite{teichnerthesis}*{Theorem~5.2.3} and the condition in this theorem is only satisfied if $|Q|>8$ and $\pi\cong P\rtimes Q$ by \cite{teichnerthesis}*{Theorems~4.4.8 and~4.4.9}. The results above shows that this is not the case in general.
\end{remark}

\begin{proof}
	(i) The first case is \cite{teichnerthesis}*{Example~5.2.4}. 
	
	(ii) If $\pi\cong P\rtimes Q$, then $\dim_{\Z/2}H^2(\pi;\Z/2)=2$ by \cite{teichnerthesis}*{Theorem~4.4.8}. By \cite{teichnerthesis}*{Proposition~5.2.2}, for $w\neq 0$, the Kirby--Siebenmann invariant is not a homotopy invariant and hence $\pi$ does not satisfy strong stable rigidity. 
	
	To show that $\pi$ satisfies stable rigidity, it suffices to show that smooth homotopy equivalent manifolds with fundamental group $\pi$ are stably homeomorphic. By \cite{teichnerthesis}*{Theorem~4.4.9}, $F_{2,2}(\xi(\pi,0))=0$ and in particular the image of $\ker(\kappa_2^h)$ in $F_{2,2}(\xi(\pi,0))$ is trivial. Hence we can restrict to the case $w\neq 0$. By \cite{teichnerthesis}*{Theorem~4.4.9}, the stable homeomorphism type is then determined by the signature and the image of the fundamental class in $H_4(\pi;\Z)$. 		
	
	(iii) If $\pi$ is not isomorphic to $P\rtimes Q$, then $H_2(\pi;\Z/2)=0$ by \cite{teichnerthesis}*{Theorem~4.4.8}. In particular, $\kappa_2^h$ is injective and $\pi$ satisfies strong stable rigidity by \cref{prop:davis-intro}. 
\end{proof}

The following can be found in \cite[Theorem 4.1(b)]{TW80} (see also \cite[Proposition 7.4]{HMTW88}).

\begin{lemma} \label{lemma:quaternionkappa2}
The map $\kappa_2^s \colon H_2(Q_{2^n};\Z/2) \to L_4^s(\Z[Q_{2^n}])$ is trivial for all $n \ge 3$. The same conclusion holds for $\kappa_2^h$.
\end{lemma}

\subsection{Dihedral and semi-dihedral groups}
\label{ss:Dih_SemiDih}

For a finite group $\pi$, let $O(\pi)$ denote the odd order normal subgroup of $\pi$ of maximal order (see, for example, \cite[Section 1.2]{Be22}). This is unique since, if $N_1, N_2 \le \pi$ are odd order normal subgroups, then $N_1 \cdot N_2 \le \pi$ is an odd order normal subgroup which contains $N_1$ and $N_2$. 
We can therefore always view $\pi$ as sitting in an extension
\[ 1 \to O(\pi) \to \pi \to \pi/O(\pi) \to 1. \]
By the Schur-Zassenhaus theorem, extensions of groups of coprime order split as semidirect products. Let $S$ denote the Sylow $2$-subgroup of $\pi$. It follows that $\pi \cong P \rtimes S$ for some group $P$ of odd order if and only if $\pi/O(\pi) \cong S$, and in this case we necessarily have that $P \cong O(\pi)$.

We will now discuss the group cohomology and extension properties of groups whose Sylow $2$-subgroup is dihedral or semi-dihedral. Different cases arise according to the number of conjugacy classes of elements of order two (i.e. involutions) and order four there are in $\pi$.

\subsubsection{Dihedral groups}

For $n\geq 3$, let $D_{2^n} = \langle x, y \mid x^2, y^2, (xy)^{2^{n-1}}\rangle$ be the dihedral group.
By \cite[IV.~Theorem~2.7]{Adem-Milgram}, we have
\[H^*(D_{2^n};\Z/2)\cong \Z/2[x,y,u]/(xy=0),\]
where $x,y$ are the $1$-dimensional classes dual to the generators in the above presentation and $u$ is $2$-dimensional with $\Sq^1(u)=(x+y)u$. 

From \cite[Theorem 5.2]{Handel} we get the calculation of the integral cohomology group of the dihedral group $D_{2^n}$ for $n\geq 3$
  \[H^*(D_{2^n};\Z)\cong \Z[a_2,b_2,c_3,d_4]/(2a_2,2b_2,2c_3,2^{k-1}d_4,b_2^2+a_2b_2,c_3^2+a_2d_4).\]

\begin{lemma}\label{lem:reductin-dihedral}
    With respect to the dual basis of $\{x^5,y^5,x^3u,y^3u,xu^2,yu^2\}$, the image of $H_5(D_{2^n},\Z)$ in $H_5(D_{2^n};\Z/2)$ is generated by $(x^5)^*,(y^5)^*,(xu^2)^*$ and $(yu^2)^*$.   
\end{lemma}

The following argument is similar to that of \cite[Lemma 2.3.7]{teichnerthesis}.

\begin{proof}
Consider the cohomological universal coefficient short exact sequence \cite[Theorem 5.5.12]{spanier_algebraic_1989} which is natural with respect to the reduction of coefficients $r_2\colon \Z\to \Z/2$. Hence we have the following commutative diagram
\[\begin{tikzcd}
	0 \arrow[r] & \Ext^1_\mathbb{Z}(H^{6}(D_{2^n};\Z);\Z) \arrow[r,"\cong"] \arrow[d,"\cong"] & H_5(D_{2^n};\mathbb{Z}) \arrow[r] \arrow[d,"r_2"]        & 0 \arrow[r] \arrow[d]                                        & 0 \\
		0 \arrow[r] & \Ext^1_\mathbb{Z}(H^{6}(D_{2^n};\Z);\Z/2) \arrow[r]         & H_5(D_{2^n};\mathbb{Z}/2) \arrow[r,"\ev"] \ar[dr,"\cong"']               & {\Hom(H^5(D_{2^n};\mathbb{Z}),\mathbb{Z}/2)} \arrow[r]             & 0\\
&&&\Hom(H^5(D_{2^n};\mathbb{Z}/2),\mathbb{Z}/2)\ar[u,"r_2^*"']&
  \end{tikzcd}\]
Here the vertical map on the left is an isomorphism since $H^6(D_{2^n};\Z)\cong (\Z/2)^4$. It follows that the image of $r_2\colon H_5(D_{2^n};\Z)\to H_5(D_{2^n};\Z/2)$ is $\ker(\ev)$ which coincides with $\ker(r_2^*)$.

We will now show that the mod 2 reduction on cohomology is given by 
\[a_2\mapsto x^2+y^2,\, b_2\mapsto x^2, \, c_3\mapsto (x+y)u, \, d_4\mapsto u^2.\] 
To see this, we will use the computations in \cite{Handel}. First fix the identification of $H^*(D_{2^n};\Z)$ given above, and the alternate identification $H^*(D_{2^n};\Z/2) \cong \Z/2[u_1,v_1,w_2]/(u_1^2+u_1v_1)$ which is related to the description above via $u_1=x$, $v_1 = x+y$ and $u=w_2$. Let $\{\iota_q^i\}$ and $\{\lambda_q^i\}$ denote the generators of the $q$th chain groups given in \cite[Section 2]{Handel} and which are used to define $H^*(D_{2^n};\Z)$ and $H^*(D_{2^n};\Z/2)$ respectively. The generators are related by $r_2(\iota_q^i) = \lambda_q^i$.
In this notation, we have $a_2 = [\iota_2^3]$, $b_2 = [\iota_2^2]$, $c_3 = [\iota_3^2]$ and $d_4 = [\iota_4^1]$ (see \cite[p.~27]{Handel}), as well as $u_1 = [\lambda_1^1]$, $v_1 = [\lambda_1^2]$ and $w_2 = [\lambda_2^1]$ (see \cite[p.~29]{Handel}).
By the product formulae for the $\lambda_q^i$ in \cite[Proposition 4.3]{Handel} (see also \cite[Proof of the Theorem 5.5]{Handel}), we obtain:
\begin{align*} 
r_2(a_2) &= r_2([\iota_2^3]) = [\lambda_2^3] = [\lambda_1^2\lambda_1^2] = v_1^2 = x^2+y^2 \\
r_2(b_2) &= r_2([\iota_2^2]) = [\lambda_2^2] = [\lambda_1^1\lambda_1^1] = u_1^2 = x^2 \\
r_2(c_3) &= r_2([\iota_3^2]) = [\lambda_3^2]= [\lambda_1^2\lambda_2^1]= v_1w_2 = (x+y)u \\
r_2(d_4) &= r_2([\iota_4^1]) = [\lambda_4^1]= [\lambda_2^1\lambda_2^1] = w_2^2 = u^2
\end{align*}
where, in each of the four products, we used cases (f), (g), (f) and (d) in \cite[Proposition 4.3]{Handel} respectively.

Hence the image of mod 2 reduction in $H^5(D_{2^n};\Z/2)$ is $\langle x^3u,y^3u\rangle$.
This implies that the image $\im(r_2\colon H_5(D_{2^n};\Z)\to H_5(D_{2^n};\Z/2))$ is generated by $(x^5)^*,(y^5)^*,(xu^2)^*,(yu^2)^*$ since it is an annihilator of $\langle x^3u,y^3u\rangle$ as claimed.
\end{proof}

The following can be found in \cite[Theorem 4.1(b)]{TW80}.

 \begin{lemma}\label{lem:dihedralKappa2}
    The map $\kappa_2^s\colon H_2(D_{2^{n}};\Z/2)\to L_4^s(\Z[D_{2^{n}}])$ is trivial for all $n \ge 3$. The same conclusion holds for $\kappa_2^h$.
   \end{lemma}

Since \cite{TW80} is unpublished, we  include an alternate proof below.

   \begin{proof}
	First note that $H_2(D_{2^{n}},\Z/2)$ is generated by its maps from $H_2(H;\Z/2)$ where $H$ ranges over the elementary abelian subgroups of $D_{2^{n}}$ \cite[Lemma 4.6]{Qu71}.
The elementary abelian subgroups of $D_{2^{n}}$ have the form $H = C_{2^k}$ for $k < n$ or $H=C_2^2$. By \cref{lem:abelian0}, we have that $\kappa_2^s(H)=0$ for all these groups. Hence $\kappa_2^s(D_{2^{n}})=0$ by \cref{prop:transfer-l-theory}. This implies that $\kappa_2^h(D_{2^n})=0$ since $\kappa_2^h$ factors through $\kappa_2^s$.
 \end{proof}
   
\begin{proposition}
    \label{prop:no-rigid-dihedral}
    For $n\geq 3$, the dihedral group $D_{2^n}$ does not satisfy stable rigidity.
\end{proposition}

The relevant James spectral sequence computation was done in \cite{Pedrotti2017}. As this is not widely available, we give a proof here.

\begin{proof}
    Let $\pi:=D_{2^n}$ and let $w:=x^2+y^2$. We have $\ker(\kappa_2^h)=H_2(\pi;\Z/2)$ by \cref{lem:dihedralKappa2}. By \cref{thmx:main-table}, it suffices to show that $E^\infty_{2,2}(\xi(\pi,w))\cong (\Z/2)^2$. 
    
    We have 
    \[\Sq^2(x^2)+x^2\smile w=\Sq^2(y^2)+y^2\smile w=0\]
    and
    \[\Sq^2(u)+u\smile w=u^2+x^2u+y^2u.\]
    Hence we have $E^3_{2,2}(\xi(\pi,w))\cong (\Z/2)^2$ generated by the images of $(x^2)^*$ and $(y^2)^*$ in dual basis of $\{x^2,y^2,u\}$. It suffices to show that $d_3\colon E^3_{5,0}(\xi(\pi,w))\to E^3_{2,2}(\xi(\pi,w))$ is trivial. 

    We have
    \[\Sq^2(x^3)+x^3\smile w=\Sq^2(y^3)+y^3\smile w=0\]
    and
    \[\Sq^2(xu)+xu\smile w=xu^2,\quad \Sq^2(yu)+yu\smile w=yu^2.\]
    Hence by \cref{lem:reductin-dihedral}, $E^3_{5,0}(\xi(\pi,w))\leq H_5(\pi;\Z)$ is generated by the images of $H_5(\Z/2;\Z)$ under the inclusions $\Z/2\to \pi$ given by $x$ and $y$ from the presentation given at the start of this subsection. Note that $w$ pulls back non-trivially to $\Z/2$ under both of these inclusions. By naturality of the James spectral sequence, it suffices to show that $d_3\colon E^3_{5,0}(\xi(\Z/2,z^2))\to E^3_{2,2}(\xi(\Z/2,z^2))$ is trivial, where $H^*(\Z/2;\Z/2)\cong \Z/2[z]$. As a map $H_2(\Z/2;\Z/2)\to \Z/2$, $z^2$ has to factor through $E^3_{2,2}(\xi(\Z/2,z^2))$ by \cref{lem:ter-ks}. In particular, $E^3_{2,2}(\xi(\Z/2,z^2))\neq 0$. Since $H_2(\Z/2;\Z/2)\cong\Z/2$, this implies that $d_3$ is trivial as needed.
\end{proof}

\begin{lemma}
\label{lemma:dihedral}
	Let $\pi$ be a finite group such that its $2$-Sylow subgroup $D_{2^n}$ is a dihedral group of order $2^n \ge 4$. Then precisely one of the following cases holds, where ccls refers to the conjugacy classes of group elements.
\FloatBarrier
\bgroup
\def\arraystretch{1.5}
\normalfont
\begin{table}[h]
\begin{center}
\begin{tabular}{|c|c|c|c|}
      \hline
     & $\substack{\text{$\#$ ccls of} \\ \text{involutions}}$ & Is $\pi/O(\pi) \cong D_{2^n}$?  & $H^2(\pi;\FF_2)$ \\
       \hline
      Case $1$   & 1 & No & $\FF_2$ \\
      \hline
      Case $2$   & 2 & No & $\FF_2^2$ \\
            \hline
      Case $3$   & 3 & Yes  & $\FF_2^3$ \\
      \hline
    \end{tabular}
\end{center}
\end{table}
\egroup
\FloatBarrier
Furthermore, in Case $2$, we have $H^*(\pi;\FF_2) \cong \FF_2[a,b,c]/(ac)$ where $|a|=1$, $|b|=2$ and $|c|=3$.
\end{lemma}

\begin{proof}
In \cite[Section 2.7]{Be22}, it is shown that finite groups $\pi$ whose Sylow $2$-subgroups are dihedral split into three cases according to number of conjugacy classes of involutions. Cases 1-3 in our table corresponds to Cases 2.7.1, 2.7.2 and 2.7.3 respectively. The information given there determines $\pi/O(\pi)$. 
In \cite{AS93}, such groups $\pi$ are split into Cases (1)-(3) according to the conjugacy of certain order two elements coming from the Sylow $2$-subgroup $D_{2^n} \le \pi$. It can be deduced from \cite[Fact 1.1]{AS93} that these cases correspond to Cases 1-3 above. The cohomology rings $H^*(\pi;\FF_2)$ are computed in each case \cite[Main Theorem]{AS93}, from which the result follows.
\end{proof}

\begin{proposition}
    Let $\pi$ be a finite group such that its $2$-Sylow subgroup $D_{2^n}$ is dihedral of order $2^n \ge 8$. Then $\pi$ satisfies stable rigidity if and only if $\pi$ is not a semi-direct product $P\rtimes D_{2^n}$.
\end{proposition}

\begin{proof}
    The group $D_{2^n}$ does not satisfy stable rigidity by \cref{prop:no-rigid-dihedral}. Hence if $\pi$ is a semi-direct product $P\rtimes D_{2^n}$ it does not satisfy stable rigidity by \cref{prop:odd-order-normal-sg}.

    We now show that $\pi$ satisfies stable rigidity if it is not a a semi-direct product $P\rtimes D_{2^n}$. By \cref{lemma:dihedral}, in this case $H_2(\pi;\Z/2)$ has dimension $1$ or $2$.

    We first consider the case where the dimension is $1$. By \cref{cor:smallf22}, it remains to show that $E^\infty_{2,2}(\xi(\pi,0))=0$. Since $H_2(D_{2^n};\Z/2)\to H_2(\pi;\Z/2)$ is surjective, this follows from $E^\infty_{2,2}(\xi(D_{2^n},0))=0$, which can be seen as follows.

    From the cohomology ring given above it can easily be verified that the Steenrod square $\Sq^2\colon H^2(D_{2^n};\Z/2)\to H^4(D_{2^n};\Z/2)$ is injective, so the dual $\Sq_2=d_2\colon H^4(D_{2^n};\Z/2)$ is surjective and hence already $E^3_{2,2}(\xi(D_{2^n},0))=0$.

    Next we consider the case where the dimension is $2$. Then the cohomology ring is given by $H^*(\pi;\Z/2) \cong \Z/2[a,b,c]/(ac)$ where $\deg(a)=1$, $\deg(b)=2$ and $\deg(c)=3$. Hence 
    \[
    \Sq_2\colon H_4(\pi;\Z/2)\to H_2(\pi;\Z/2)
    \]
    is surjective and $E^3_{2,2}(\xi(\pi,0))=0$. By \cref{cor:smallf22}, it remains to show that $E^\infty_{2,2}(\xi(\pi,w))\cong \Z/2$ for $w\neq 0$. It is easy to see that for all choices of $w\neq 0$, the map 
    \[
    H^2(\pi;\Z/2)\to H^4(\pi;\Z/2), \quad x\mapsto \Sq^2(x)+x\smile w
    \]
    is non-trivial. Hence the second differential is non-trivial and $E^\infty_{2,2}(\xi(\pi,w))\cong \Z/2$ as needed.
    \end{proof}
    
\subsubsection{Semi-dihedral groups}

We define the \textit{semi-dihedral group of order} $2^k$ to be
\[SD_{2^k}=\left<x,y| x^{2^{k-1}}=y^2=1, yxy= x^{2^{k-2}-1}\right>\]
with $k\geq 4$.

\begin{lemma} \label{lemma:semidihedralkappa2}
The map $\kappa_2^s\colon H_2(SD_{2^{n}};\Z/2)\to L_4^s(\Z[SD_{2^{n}}])$ is trivial for all $n \ge 4$. The same conclusion holds for $\kappa_2^h$.
\end{lemma}

\begin{proof}
    According to \cite{Chin1995} the $\Z/2$ cohomology of $SD_{2^k}$ is detected by $D_8$ and $Q_8$, the dihedral and quaternionic groups of order $8$, i.e.\ $H^*(SD_{2^k};\Z/2)\hookrightarrow H^*(D_8;\Z/2)\oplus H^*(Q_8;\Z/2)$ is injective. It follows by duality that $H_2(D_8;\Z/2)\oplus H_2(Q_8;\Z/2)\twoheadrightarrow H_2(SD_{2^k};\Z/2)$ is surjective. So any element in $H_2(SD_{2^k};\Z/2)$ comes from $H_2(D_8;\Z/2)\oplus H_2(Q_8;\Z/2)$. We know that $\kappa^s_2$ vanishes on dihedral and quaternionic groups. By naturality, $\kappa^s_2$ vanishes on $SD_{2^k}$.
\end{proof}

\begin{proposition}
\label{prop:non-rigid-semi}
The group $SD_{2^k}$ does not satisfy stable rigidity.    
\end{proposition}

\begin{proof}
    For this proof we will use the presentation $D_{2n}=\langle x',y'\mid x'^n=y'^2=1,y'x'y'=x'^{-1}\rangle$ for the dihedral group. Define a map $f\colon SD_{2^k}\to D_{2^{k-1}}$ by sending $x\mapsto x'$ and $y\to y'$. By \cite[Lemma 3, p. 70]{Evens1985}, we have
\[H^*(SD_{2^k};\Z/2)\cong  \Z/2[x,y,P,u_3]/(x^2=xy, x^3=0, u_3x=0, u_3^2=P(x^2+y^2)),\]
where $\deg(x)=\deg(y)=1$, $\deg(u_3)=3$ and $\deg(P)=4$. 
Let $w\in H^2(D_{2^{k-1}};\Z/2)$ as in the proof of \cref{prop:no-rigid-dihedral}.
Since $f^*(x')=x$ and $f^*(y')=y$ and $E_{2,2}^3(D_{2^{k-1}},w)$ is generated by $(x^2)^*$ and $(y^2)^*$, the induced map
\[f_*\colon E_{2,2}^3(SD_{2^k},f^*w)\to E_{2,2}^3(D_{2^{k-1}},w)\]
is surjective. Since $E_{2,2}^\infty(D_{2^{k-1}},w)\cong (\Z/2)^2$, also $E_{2,2}^\infty(SD_{2^{k}},f^*w)=H_2(SD_{2^{k}};\Z/2)\cong (\Z/2)^2$.

Hence there exists $x\in \ker(\kappa_2^h)=H_2(SD_{2^k};\Z/2)$ such that $0\neq [x]\in E^\infty_{2,2}(\xi(SD_{2^k},0,f^*w))$ and $f^*w\frown x= 0$. Hence $SD_{2^k}$ does not satisfy stable rigidity by \cref{prop:stable-he-vs-homeo}.
\end{proof}

\begin{lemma}
\label{lemma:semidihedral}
	Let $\pi$ be a finite group such that its $2$-Sylow subgroup $SD_{2^n}$ is semi-dihedral of order $2^n \ge 16$. Then precisely one of the following cases holds, where ccls refers to the conjugacy classes of group elements.
\FloatBarrier
\bgroup
\def\arraystretch{1.5}
\normalfont
\begin{table}[h]
\begin{center}
\begin{tabular}{|c|c|c|c|c|}
      \hline
                 & $\substack{\text{$\#$ ccls of} \\ \text{involutions}}$ & $\substack{\text{$\#$ ccls of order} \\ \text{four elements}}$ & Is $\pi/O(\pi) \cong SD_{2^n}$?  & $H^2(\pi;\FF_2)$ \\
       \hline
      Case $1$   & 1 & 1 & No & $0$ \\
      \hline
      Case $2$   & 2 & 1 & No  & $\FF_2$ \\
      \hline
      Case $3$   & 1 & 2 & No & $\FF_2$ \\
            \hline

      Case $4$   & 2 & 2 & Yes  & $\FF_2^2$ \\
      \hline
    \end{tabular}
\end{center}
\end{table}
\egroup
\FloatBarrier
\end{lemma}

To prove this, we use the Alperin–Brauer–Gorenstein theorem \cite{ABG70} which classifies the finite simple groups whose Sylow $2$-subgroup is $SD_{2^n}$.
The classification is according to the possible fusion systems on $SD_{2^n}$.
That is, for a subgroup $\pi' \le \pi$, we will be concerned with the value of $\Aut_{\pi}(\pi') := N_{\pi}(\pi')/C_{\pi}(\pi')$ where $N$ and $C$ are the normaliser and centraliser respectively.

\begin{proof}
Similarly to the case of dihedral groups, results in \cite[Section 3.5]{Be22} show that finite groups $\pi$ whose Sylow $2$-subgroups are semi-dihedral split into Cases 3.5.1, 3.5.2, 3.5.3, 3.5.4 according to the number of conjugacy classes of involutions and order four elements. This corresponds to Cases 1-4 above. The information given there similarly determines $\pi/O(\pi)$.
In \cite{Ch95}, such groups are split into Cases (1)-(4) and $H^2(\pi;\FF_2)$ is determined for the groups in each class. It therefore remains to find how Cases (1)-(4) match up with Cases 1-4 in our table. 
In \cite{Ch95}, the cases are determined according to whether $3 \mid |\Aut_{\pi}(C_2^2)|$ and whether $3 \mid |\Aut_{\pi}(Q_8)|$ for naturally defined subgroups $C_2^2, Q_8 \le SD_{2^n} \le \pi$. 
By \cite[Proposition 1]{ABG70}, we have $|\Aut_{\pi}(C_2^2)|=|\Aut_{\pi}(Q_8)|=6$ in Case 1, $|\Aut_{\pi}(C_2^2)|=2$, $|\Aut_{\pi}(Q_8)|=6$ in Case 2, $|\Aut_{\pi}(C_2^2)|=6$, $|\Aut_{\pi}(Q_8)|=2$ in Case 3, and $|\Aut_{\pi}(C_2^2)|=|\Aut_{\pi}(Q_8)|=2$ in Case 4.
Hence Cases 1-4 corresponds to Cases (1), (3), (2) and (4) in \cite{Ch95}. The result follows.
\end{proof}

\begin{proposition}
    Let $\pi$ be a finite group such that its $2$-Sylow subgroup $SD_{2^n}$ is semi-dihedral of order $2^n \ge 16$. Then $\pi$ satisfies stable rigidity if and only if $\pi$ is not a semi-direct product $P\rtimes SD_{2^n}$.
\end{proposition}

\begin{proof}
    The group $SD_{2^n}$ does not satisfy stable rigidity by \cref{prop:non-rigid-semi}. Hence if $\pi$ is a semi-direct product $P\rtimes SD_{2^n}$ it does not satisfy stable rigidity by \cref{prop:odd-order-normal-sg}.

    We now show that $\pi$ satisfies stable rigidity if it is not a a semi-direct product $P\rtimes SD_{2^n}$. By \cref{lemma:semidihedral}, in this case $H_2(\pi;\Z/2)$ has dimension $0$ or $1$.
    By \cref{cor:smallf22}, it remains to show that $E^\infty_{2,2}(\xi(\pi,0))=0$. Since $H_2(SD_{2^n};\Z/2)\to H_2(\pi;\Z/2)$ is surjective, this follows from $E^\infty_{2,2}(\xi(SD_{2^n},0))=0$, which can be seen as follows.

    Recall that by \cite{Evens1985} the cohomology ring of $SD_{2^n}$ is
    \[H^*(SD_{2^n};\Z/2)\cong \Z/2[x,y,P,u_3]/(x^2=xy,x^3=0,u_3x=0,u_3^2=P(x^2+y^2)).\]
    Hence $\Sq^2(x^2)=0$ and $\Sq^2(y^2)=y^4\neq 0$. By \cite{Evens1985}, there is a map $\psi\colon Q_8\to SD_{2^n}$ such that $\psi^*(x^2)=\psi^*(y^2)=\wt{y}^2$, where $H^*(Q_8;\Z/2)\cong \Z/2[\wt{x},\wt{y},\wt{P}]/(\wt{x}^2+\wt{x}\wt{y}+\wt{y}^2, \wt{x}^2\wt{y}+\wt{x}\wt{y}^2)$. It follows that $E^3_{2,2}(\xi(Q_8,0))$ surjects onto $E^3_{2,2}(\xi(SD_{2^n},0))\cong \Z/2$. By \cite{teichnerthesis}*{Proposition~4.2.1}, the differential $d_3\colon E^3_{5,0}(\xi(Q_8,0))\to E^3_{2,2}(\xi(Q_8,0))$ is an isomorphism. By naturality of the James spectral sequence, the differential $d_3\colon E^3_{5,0}(\xi(SD_{2^n},0))\to E^3_{2,2}(\xi(SD_{2^n},0))$ is surjective. Hence $E^\infty_{2,2}(\xi(SD_{2^n},0))=0$ as needed.
\end{proof}

\subsection{Modular maximal-cyclic groups}
\label{ss:modular-maximal}

Next note that, for all $k \ge 4$, there are four isomorphism classes of extensions of $C_2$ with kernel $C_{2^{k-1}}$ (see, for example, \cite[Chapter 5, Exercise 17]{DF04}). The groups are $C_{2^{k-1}} \times C_2$, $D_{2^k}$, $SD_{2^k}$ or the following group:
\[
M_k(2)=\left<x,y| x^{2^{k-1}}=y^2=1, yxy=x^{2^{k-2}+1}\right>
\]
which is defined for $k\geq 4$. This is known as \emph{modular maximal cyclic group of order} $2^k$. 

\begin{proposition}
\label{prop:semi-dihedral2}
    The group $M_k(2)$ satisfies stable rigidity.
\end{proposition}

\begin{proof}
    There is a surjection $f\colon M_k(2) \to C_{2^{k-2}}\times C_2=\langle a,b\mid a^{2^{k-2}},b^2,[a,b]\rangle$ given by $x\mapsto a$ and $y\mapsto b$. Using the presentations as a start for a free resolution of $\Z$ as a $\Z[M_k(2)]$- and $\Z[C_{2^{k-2}}\times C_2]$-module respectively, we can consider the relations as generators for the second homology. Since $k\geq 4$, $\kappa_2^h$ is non-trivial for $C_{2^{k-2}}\times C_2$ as in the proof of \cref{lem:abelian1}. It follows that it is non-trivial on the homology class represented by $[a,b]$ and trivial on the classes represented by $a^{2^{k-2}}$ and $b^2$ since $\kappa_2^h$ vanishes on $C_2$ and on $C_{2^{k-2}}$. Under the map $f$ the relation $yxy=x^{2^{k-2}+1}$ corresponds to $[a,b]=a^{2^{k-2}}$. And thus the homology class represented by $yxy=x^{2^{k-2}+1}$ has non-trivial image under $\kappa_2^h$.

    Now consider the inclusion $g\colon C_{2^{k-2}}\times C_2=\langle a,b\mid a^{2^{k-2}},b^2,[a,b]\rangle\to M_k(2)$ given by $a\mapsto x^2$ and $b\mapsto y$. It maps the homology classes represented by $a^{2^{k-2}}$ and $b^2$ to those represented by $x^{2^{k-1}}$ and $y^2$, respectively. Thus $g$ induces a surjection on $\ker(\kappa_2^h)$. Since $C_{2^{k-2}}\times C_2$ satisfies stable rigidity, so does $M_k(2)$ by \cref{lem:surj-ker-kappa}.
\end{proof}

Combining \cref{prop:semi-dihedral2} with \cref{prop:oddindex} we obtain the following result.

\begin{corollary}
    Let $\pi$ be a finite group with $2$-Sylow subgroup $M_k(2)$, then $\pi$ satisfies stable rigidity.
\end{corollary}

\begin{lemma} \label{lemma:kappa_SD'}
    For $M_k(2)$, we have $\ker(\kappa_2^s)=\ker(\kappa_2^h)$.
\end{lemma}

\begin{proof}
    Since $\kappa_2^s$ is trivial for cyclic groups by \cite[Theorem 4.1(b)]{taylor-williams}, $\kappa_2^s$ is trivial by the homology classes represented by the relations $x^{2^{k-1}}=1$ and $y^2=1$. As in the proof of \cref{prop:semi-dihedral2}, $\kappa_2^h$ is non-trivial on the homology class represented by $yxy=x^{2^{k-2}+1}$. Since $\ker(\kappa_2^s)\leq \ker(\kappa_2^h)$, the claim follows.
\end{proof}

\section{Comparison of homotopy and simple homotopy up to stabilisations} \label{s:he-vs-she}

We will now prove the following theorem from the introduction

\begin{reptheorem}{cor:he-vs-she}
There exist closed oriented smooth $4$-manifolds $M$, $M'$ that are homotopy equivalent but not simple homotopy equivalent up to stabilisations if and only if there exists a finitely presented group $\pi$ with $\ker(\kappa_2^s)\neq\ker(\kappa_2^h) \subseteq H_2(\pi;\Z/2)$.
\end{reptheorem}

\begin{proof}
	If $\ker(\kappa_2^s)=\ker(\kappa^h_2)$ for a group $G$, then by \cref{thmx:main-table} manifolds with fundamental group $G$ are homotopy equivalent up to stabilisation if and only if they are simple homotopy equivalent up to stabilisation. 
	
	Now assume there exists $x\in \ker(\kappa^h_2(G))$ which is not in the kernel of $\kappa_2^s(G)$. Let $w\in H^2(G;\Z/2)\cong H^2(G;\Z/2)^*$ be such that $w(x)\neq 0$ and $w(y)=0$ for every $y\in \ker(\kappa_2^s(G))$. Let $\pi:=G*G$ and $\xi:=\xi(\pi,(w,w))$.

 We will first show that $(x,x)\in H_2(\pi;\Z/2)$ is in the kernel of $\kappa^h_2(\pi)$. It suffices to show that $(x,0)$ and $(0,x)$ are in the kernel. The map $H_2(G;\Z/2)\to H_2(\pi;\Z/2)$ induced by the inclusion $G\to G*G=\pi$ sends $x$ to $(x,0)$. Since $x\in \ker(\kappa_2^h(G))$, $(x,0)$ is in the kernel of $\kappa_2^h(\pi)$ by naturality of $\kappa_2^h$. Similarly, $(0,x)$ is in the kernel as claimed.
 
    Next we show that $(x,x)\neq 0\in E^\infty_{2,2}/[\ker \kappa_2^s(\pi)]$. Since $x\notin\ker(\kappa_2^s(G))$, $(x,x)\notin \ker(\kappa_2^s(\pi))$ by naturality of $\kappa^s_2$. By naturality of the James spectral sequence $(a,a')\in H_4(\pi;\Z/2)\cong H_4(G;\Z/2)\oplus H_4(G;\Z/2)$ maps to $(d_2^G(a),d_2^G(a'))\in H_2(\pi;\Z/2)\cong H_2(G;\Z/2)\oplus H_2(G;\Z/2)$ under the differential $d_2^\pi$, where $d_2^G$ is the differential in $E_{p,q}^2(\xi(G,w))$. Similarly, $d_3^\pi=(d_3^G,d_3^G)$. Hence it suffices to show that $x$ is not contained in $\kappa_2^s(\pi)+\im d_2,d_3$. By \cref{lem:ter-ks} and the definition of $w$, $\kappa_2^s(\pi)+\im d_2,d_3$ is contained in $\ker w$, while $x\notin \ker w$. Hence $(x,x)\neq 0\in E^\infty_{2,2}/[\ker \kappa_2^s(\pi)]$ as claimed.
     
 Let $M$ be a manifold with normal $1$-type $\xi$, some $1$-smoothing $\wt{\nu}$, $\sigma(M)=0$, $\pri(M)=\sec(M)=0$ and $\ter(M)=(x,x)$. Let $M'$ be a manifold with normal $1$-type $\xi$ that is $\xi$-nullbordant. Then $M'$ and $M$ are not stably diffeomorphic since $\ter(M)\neq0$ while $\ter(M')=0$. According to \cref{rm:aboutXiActionOnOmega}, regardless of the choice of $1$-smoothings we have $\ter([M]-[M'])=(x,x)\notin [\ker(\kappa_2^s(\pi))]$. So they are also not simple homotopy equivalent after stabilisation by \cref{thmx:main-table}. However, they are homotopy equivalent after stabilisation since $\ter([M]-[M'])=(x,x)\in [\ker(\kappa^h_2(\pi))]$.
	
	 By \cref{lem:ter-ks}, $\ks(M')=0$ and also
    \[\ks(M)=(w,w)(x,x)=w(x)+w(x)=1+1=0.\] 
    Hence after some number of stabilisations, $M$, $M'$ become smoothable and there is a pair of smooth, homotopy equivalent manifolds that are not simple homotopy equivalent up to stabilisation.
\end{proof}

The following is the direct analogue of \cref{prop:oddindex}. The proof is essentially the same, but is repeated here for the convenience of the reader.

\begin{proposition} \label{prop:oddindex-k2}
Let $G \le \pi$ be a subgroup of odd index. If $\ker(\kappa_2^h(G)) =\ker(\kappa_2^s(G))$, then $\ker(\kappa_2^h(\pi)) =\ker(\kappa_2^s(\pi))$.
\end{proposition}

\begin{proof}
    Not that $\ker(\kappa_2^s(\pi))\leq \ker(\kappa_2^h(\pi))$. Hence it remains to show the other inclusion. Let $i\colon G\to\pi$ be the inclusion and let $x\in H_2(\pi;\Z/2)$ be in the kernel of $\kappa_2^h(\pi)$. Let $x':=\tr(x)\in H_2(G;\Z/2)$, where $\tr$ is the transfer map. Since $G$ has odd index, $\tr\circ i_*$ is the identity and thus $x'$ is a preimage of $x$. By \cref{prop:transfer-l-theory}, $\kappa_2^h(x')=p^*(\kappa^h_2(x))=0$. Hence $x'\in \ker(\kappa_2^h(G)) =\ker(\kappa_2^s(G))$. By naturality of $\kappa_2^s$, $x$ is in the kernel of $\kappa_2^s(\pi)$ as needed.
\end{proof}

We will now establish the following result which was stated previously in the introduction.

\begin{reptheorem}{prop:k_2-equal}
Let $\pi$ be a finitely generated abelian group or a finite group whose Sylow $2$-subgroup is abelian, basic, modular maximal-cyclic, or of order at most $16$.
Then $\ker(\kappa_2^s) = \ker(\kappa_2^h)$. In particular, closed, oriented, smooth $4$-manifolds $M$, $M'$ with fundamental group $\pi$ which are homotopy equivalent are simple homotopy equivalent up to stabilisations.
\end{reptheorem}

\begin{proof}
The case of finitely presented abelian groups is covered by \cref{lem:abelian0}. Now restrict to the case where $\pi$ is finite. By \cref{prop:oddindex-k2}, it suffices to restrict to finite $2$-groups.
If $\pi$ is quaternionic, dihedral or semi-dihedral, then $\ker(\kappa_2^s)=\ker(\kappa_2^h) = H_2(\pi;\Z/2)$ by \cref{lemma:quaternionkappa2,lem:dihedralKappa2,lemma:semidihedralkappa2}. If $\pi$ is $M_k(2)$, then $\ker(\kappa_2^s)=\ker(\kappa_2^h)$ by \cref{lemma:kappa_SD'}.
It remains to prove that $\ker(\kappa_2^s) = \ker(\kappa_2^h)$ for all finite $2$-groups $\pi$ with $|\pi| \le 16$.

From the tables in GroupNames \cite{groupnames}, there are exactly five groups of order 2, 4, 8 or 16 which are not abelian, quaternion, dihedral, semi-dihedral or modular maximal-cyclic. Hence it remains to show that $\ker(\kappa_2^h) = \ker(\kappa_2^s)$ for these groups. The groups, which all have order 16, are $D_8 \times C_2$, $Q_8 \times C_2$, $G_{(16,3)} = C_2^2 \rtimes C_4$, $G_{(16,4)} = C_4 \rtimes C_4$ and $G_{(16,13)} = C_4 \circ D_8$ where $G_{(n,m)}$ denotes the $m$th group of order $n$ in GAP's Small Groups library \cite{GAP4}.

Let $\pi$ be one of these five groups. Using the HAP package \cite{HAP} in GAP,  
we can compute $H_2(\pi;\Z/2)$. If $i \colon G \hookrightarrow \pi$ is a subgroup inclusion, then we can also compute the image of the induced map $i_* \colon H_2(H;\Z/2) \to H_2(\pi;\Z/2)$. By \cref{lem:abelian0,lemma:quaternionkappa2,lem:dihedralKappa2}, the groups $C_2$, $C_4$, $C_2^2$, $C_8$, $D_8$, $Q_8$ and $C_2^3$ have $\kappa_2^h=\kappa_2^s=0$, i.e. all groups of orders $2$, $4$ and $8$ except $C_2 \times C_4$.
Let $H_2(\pi;\Z/2)' \le H_2(\pi;\Z/2)$ denote the subgroup generated by the inclusions of any of these groups. We then obtain the following using GAP.
\FloatBarrier
\bgroup
\def\arraystretch{1.5}
\normalfont
\begin{table}[h]
\begin{center}
\begin{tabular}{|c|c|c|c|c|c|}
      \hline
      $\pi$ & $D_8 \times C_2$ & $Q_8 \times C_2$ & $G_{(16,3)}$ & $G_{(16,4)}$ & $G_{(16,13)}$ \\ 
       \hline
      $H_2(\pi;\Z/2)$ & $(\Z/2)^6$ & $(\Z/2)^5$ & $(\Z/2)^4$ & $(\Z/2)^3$ & $(\Z/2)^5$  \\
      \hline
      $H_2(\pi;\Z/2)'$ & $(\Z/2)^6$ & $(\Z/2)^5$ & $(\Z/2)^4$ & $(\Z/2)^2$ & $(\Z/2)^5$ \\
      \hline
    \end{tabular}
\end{center}
\end{table}
\egroup
\FloatBarrier
Since $\kappa_2^X = \kappa_2^X(\pi)$ is functorial in $\pi$ for each $X \in \{h,s\}$, it follows that $H_2(\pi;\Z/2)' = H_2(\pi;\Z/2)$ implies that $\kappa_2^h = \kappa_2^s = 0$. Hence the above table implies that, if $\pi \ne G_{(16,4)}$, then $\kappa_2^h = \kappa_2^s = 0$ and so $\ker(\kappa_2^h) = \ker(\kappa_2^s)$.

Let $\pi = G_{(16,4)}$. By \cite{groupnames}, we have $\pi = \langle x, y \mid x^4, y^4, yxy^{-1}x \rangle$. Consider the surjective group homomorphism 
\[
f \colon \pi = \langle x, y \mid x^4, y^4, yxy^{-1}x \rangle \twoheadrightarrow \langle x,y \mid x^2, y^2, [x,y] \rangle = C_2 \times C_4.
\]
The homology class in $H_2(\pi;\Z/2)$ represented by $yxy^{-1}x$ maps onto the class $c \in H_2(C_2 \times C_4;\Z/2)$ represented by $[x,y]$. By \cref{lem:abelian0}, we have that $\kappa_2^h(c) \ne 0$. It follows that $\kappa_2^h(\pi) \ne 0$. Since $H_2(\pi;\Z/2)' \le H_2(\pi;\Z/2)$ has index two, it follows that $\ker(\kappa_2^h) = \ker(\kappa_2^s) = (\Z/2)^2$.
\end{proof}

\section{Comparison of stable and unstable equivalence relations }
\label{sec:cancellation}

We will now prove the following Theorem from the introduction.

\begin{reptheorem}{thm:decomp}
Let $M$, $N$ be closed oriented $4$-manifolds with good fundamental groups. If $M$, $N$ are (simple) homotopy equivalent up to stabilisations, then there exists a $4$-manifold $N'$ which is (simple) homotopy equivalent to $N$ and stably homeomorphic to $M$.
\end{reptheorem}

The proof will make use of the following consequence of Freedman's disc theorem \cite{Fr83}.

\begin{lemma} \label{prop:stable-summand}
	Let $M$ be a closed oriented $4$-manifold whose fundamental group is good. If there exists a $4$-manifold $N$ such that $M \simeq N \# (S^2 \times S^2)$ (resp. $M \simeq_s N \# (S^2 \times S^2)$), then there exists a $4$-manifold $N'$ such that $M \homeo N' \# (S^2 \times S^2)$ and $N \simeq N'$ (resp. $N \simeq_s N'$).
\end{lemma}

\begin{proof}
    Let $f\colon N\#(S^2\times S^2)\to M$ be a homotopy equivalence. If follows from Freedman's disc theorem in the oriented case \cite[p.~647]{Fr83} (see also \cite[Theorem 2.3]{PRT21}) that $f|_{S^2\times S^2\setminus D^4}$ is homotopic to an embedding. Hence there exists a $4$-manifold $N'$ and a homeomorphism $g \colon M \to N' \# (S^2 \times S^2)$. It remains to show that we have $N\simeq N'$ (resp. $N \simeq_s N'$)
    
    By homotopy extension, we can assume that $f|_{S^2\times S^2\setminus D^4}$ is an embedding. Let $p_1,p_2\colon S^2\to N\#(S^2\times S^2)$ be the inclusions of the $S^2$ factors. Then there is a homotopy equivalence
    \[ F\colon N \xrightarrow[]{\simeq} (N\#(S^2\times S^2))\cup_{p_1,p_2}(D^3\sqcup D^3)\xrightarrow{f \cup \id}(N'\#(S^2\times S^2))\cup_{p_1',p_2'}(D^3\sqcup D^3) \xrightarrow[]{\simeq} N'\]
where $f \cup \id$ is the map extended to the identity on $D^3 \sqcup D^3$ and $p_i' = g \circ f \circ p_i$.

Next suppose that $f\colon N\#(S^2\times S^2)\to M$ is a simple homotopy equivalence, i.e. it is homotopic to a sequence of elementary expansions and contractions of cells. Then the homotopy equivalence $F$ above is simple. To see this, note that the first and last homotopy equivalences are just expansions and contractions of cells. The middle homotopy equivalence $f \cup \id$ is simple since, by \cite[Theorem 1]{Wa66}, the expansions and contractions which comprise $f$ can be taken to fix $S^2\times S^2\setminus D^4$ and so $f \cup \id$ is also homotopic to a sequence of expansions and contractions.
\end{proof}

\begin{proof}[Proof of \cref{thm:decomp}]
Suppose $N_0:=M \# a(S^2 \times S^2)$ and $N \# b(S^2 \times S^2)$ are (simple) homotopy equivalent, where $a, b \ge 0$. 
By applying \cref{prop:stable-summand} to $N_0$ and $N \# (b-1)(S^2 \times S^2)$, we get that there exists a $4$-manifold $N_1$ such that $N_0 \homeo N_1 \# (S^2 \times S^2)$ and $N_1$, $N \# (b-1)(S^2 \times S^2)$ are (simple) homotopy equivalent. By applying the same argument inductively, we obtain $4$-manifolds $N_i$ for $2 \le i \le b$ such that $N_{i-1} \homeo N_i \# (S^2 \times S^2)$ and $N_i$, $N \# (b-i)(S^2 \times S^2)$ are (simple) homotopy equivalent for all $1 \le i \le b$. Let $N' = N_b$. Then $M \# a(S^2 \times S^2) \homeo N' \# b(S^2 \times S^2)$ and $N'$, $N \# b(S^2 \times S^2)$ are (simple) homotopy equivalent.
\end{proof}

We will also give the following alternative proof using \cref{thm:main-top} (the topological version of \cref{thmx:main-table}). The argument still relies on work of Freedman, but it is applied in a different way.

\begin{proof}[Proof of \cref{thm:decomp} (alternative)]
 	If $N$ is totally non-spin, then $M$ and $N$ are already stably homeomorphic. Hence we can assume that $N$ is almost spin.
	
	Let $\xi\colon B\to \BSTop$ be the normal $1$-type of $N$. Since $M$ and $N$ are (simple) stably homotopy equivalent there exist $x\in \ker(\kappa_2^h)$ (resp. $x \in \ker(\kappa_2^s)$) and $\xi$-structures on $M$ and $N$, such that $[N]-[M]=[x]\in\Omega_4(\xi)$ by \cref{thm:main-top}. Let $x'\in H_2(N;\Z/2)$ be a lift of $x$. Since the surgery obstruction of $(0,x')\in \Z\oplus H_2(N;\Z/2)\cong \cN(N)$ is trivial by assumption on $x$, the fact that $\pi = \pi_1(N)$ is a good group implies that there exists a (simple) homotopy equivalence $f\colon N'\to N$ with $\eta(f)=(0,x')$ \citelist{\cite{Freedman82}*{Theorem~1.2}\cite{OPR21}*{Section~22.1.4}}. By \cref{{thm:davis-general}}, there exist $\xi(N,w_2)$-structures on $N'$ and $N$ which induce $\xi$-structures such that $[N']-[N]=[x]\in\Omega_4(\xi)$. Hence there are $\xi$-structures on $M$ and $N'$ such that $[M]-[N']=0$. It follows that $M$ and $N'$ are stably homeomorphic as needed.
\end{proof}

\subsection{Homotopy stable classes of $4$-manifolds}

Before moving on to our main application of \cref{thm:decomp} (\cref{thm:main-cancellation} below), we first give a reformulation of \cref{thm:decomp} in terms of manifold sets.
Let $M$ be a closed oriented $4$-manifold and define its \textit{homotopy stable class} to be
\[ \M^{\ttop,\st}_h(M) = \{\text{\normalfont $4$-manifolds } N \mid N \homeo^{\st} M \}/ \simeq. \]
This is the analogue of the set which was studied in the smooth category in \cite{CCPS22} (and denoted $\mathcal{S}^{\st}_h(M)$). 
We can similarly define versions of this set for (simple) homotopy equivalence
\[ \M^{h,\st}_h(M) = \{\text{\normalfont $4$-manifolds } N \mid N \simeq^{\st} M \}/ \simeq, \quad \M^{s,\st}_h(M) = \{\text{\normalfont $4$-manifolds } N \mid N \simeq_s^{\st} M \}/ \simeq. \]

The following is a direct consequence of \cref{thm:decomp}.

\begin{corollary} \label{cor:homotopy-stable-class}
 Let $M$ be a closed oriented $4$-manifold whose fundamental group $\pi$ is good. Then the natural inclusion maps
 $\M^{\ttop,\st}_h(M) \xrightarrow[]{\cong} \M^{s,\st}_h(M) \xrightarrow[]{\cong} \M^{h,\st}_h(M)
$
are bijections.
\end{corollary}

\begin{remark}
Since (simple) homotopy equivalence up to stabilisations extends in a reasonable manner to the setting of finite Poincar\'{e} $4$-complexes, we have that 
\[ \{\text{\normalfont $4$-manifolds } N \mid N \simeq^{\st} M \}/ \simeq  \,\, \hookrightarrow \,\, \{\text{\normalfont finite $\text{\normalfont PD}_4$-complexes } N \mid N \simeq^{\st} M \}/ \simeq. \]
Over good fundamental groups, \cref{cor:homotopy-stable-class} therefore gives an approach to studying $\M^{\st}_h(M)$ by first computing its analogue for finite Poincar\'{e} $4$-complexes, and then determining which of these complexes are homotopy equivalent to closed $4$-manifolds.
\end{remark}

\subsection{Cancellation problems for $4$-manifolds}

We will now use \cref{thm:decomp} to establish a relationship between the cancellation problems for homeomorphism and (simple) homotopy equivalence. 
We will begin with the following definition. 

\begin{definition}
    The \emph{topological genus} $g^{\ttop}(M)$ of a closed oriented $4$-manifold $M$ is the maximal number $k$ such that there exists a $4$-manifold $M_0$ with $M\homeo M_0\#k(S^2\times S^2)$. 
    
    Similarly, define the \emph{(simple) homotopy genus} $g^{h}(M)$ (resp. $g^{s}(M)$) to be the maximal number $k$ such there exists a $4$-manifold $M_0$ with $M\simeq M_0\#k(S^2\times S^2)$ (resp. $M\simeq_s M_0\#k(S^2\times S^2)$). 
\end{definition}

By \cref{prop:stable-summand}, we have that $g^h(M) = g^s(M) = g^{\ttop}(M)$ if $M$ is a closed oriented $4$-manifold whose fundamental group is good.

\begin{definition}
    The \emph{cancellation bound} $\cb(\pi)$ for a finitely presented group $\pi$ is the minimal number $k$ such that for every closed, oriented $4$-manifold $M$ of topological genus $k$ and with fundamental group $\pi$ stable homeomorphism implies homeomorphism. That is, 
    \[\cb(\pi) = \min\{ \, k  \,\, \mid \,\, \text{\normalfont\normalsize if $M \homeo^{\st} N$, $\pi_1(M) \cong \pi$ and $g^{\ttop}(M)=k$, then $M \homeo N$} \, \}.
    \]
    We set $\cb(\pi) = \infty$ if no such $k$ exists, i.e. if the above set is empty.

    The corresponding bounds with homeomorphism  replaced by (simple) homotopy equivalence will be denoted by $\cb^s(\pi)$ and $\cb^h(\pi)$ respectively.
\end{definition}

Note that the analogue in the smooth category was referred to as the cancellation genus in \cite[Problem 10]{BC+21}. 

\begin{remark}
It is not currently known whether there exists a finitely presented group $\pi$ for which $\cb(\pi) = \infty$, or even with $\cb(\pi) \ge 2$. 
The bounds $\cb^h(\pi)$ and $\cb^s(\pi)$ have analogues in the case of finite $2$-complexes with $- \# (S^2 \times S^2)$ replaced by $- \vee S^2$.
For each $k \ge 2$, examples where the bounds are at least $k$ were constructed in \cite[Theorem B]{Ni23} (and with bound $\infty$ in the non-finite case \cite[Theorems C \& 9.2]{Ni23}).

From these finite $2$-complexes, it is possible to construct closed oriented smooth $4$-manifolds by taking the boundary of thickening in $\R^5$. Using this, it is shown in \cite[Theorem 1.4]{Ni23} that there would exists a group $\pi$ with the manifold cancellation bound $\cb(\pi) \ge k$ for each $k$ provided certain stably free $\Z \pi$-modules are not free. Since the manifolds are smooth, such examples would also give that $\cb^{\diff}(\pi) \ge k$ for the version of the bound in the smooth category. 
\end{remark}

Our main result on cancellation is the following.

\begin{theorem} \label{thm:main-cancellation}
Let $\pi$ be a good group. Then $\cb^h(\pi) \le \cb^s(\pi) \le \cb(\pi)$.
\end{theorem}

\begin{proof}
We will start by proving that $\cb^s(\pi) \le \cb(\pi)$. If $\cb(\pi)=\infty$, then there is nothing to prove, so suppose $\cb(\pi) = k$ for some $k \ge 0$.
Let $M$, $N$ be closed oriented topological $4$-manifolds with fundamental groups $\pi$ such that $M \simeq_s^{\st} N$ and $M \simeq_s M_0 \# k(S^2 \times S^2)$. 
By \cref{thm:decomp}, there exists a $4$-manifold $N'$ such that $M \homeo^{\st} N'$ and $N \simeq_s N'$.
By repeated application of \cref{prop:stable-summand}, we get that there exists a $4$-manifold $M_0'$ such that $M \homeo M_0' \# k(S^2 \times S^2)$.
Since $\cb(\pi) = k$, this implies that $M \homeo N'$. Since $N \simeq_s N'$, this implies that $M \simeq_s N$. In particular, $\cb^s(\pi) \le k$ as required.

We can similarly obtain $\cb^h(\pi) \le \cb^s(\pi)$. Suppose $M \simeq^{\st} N$ and $M \simeq M_0 \# k(S^2 \times S^2)$. Firstly, \cref{thm:decomp} implies that there exists $N'$ such that $M \homeo^{\st} N'$ and $N' \simeq N$. In particular, we have $M \simeq_s^{\st} N'$. Secondly, \cref{prop:stable-summand} implies that there exists $M_0'$ such that $M \homeo M_0' \# k(S^2 \times S^2)$. In particular, we have $M \simeq_s M_0' \# k(S^2 \times S^2)$. So, if $\cb^s(\pi) = k$, then the same argument as before implies that $\cb^h(\pi) \le k$, as required.
\end{proof}

\begin{proof}[Proof of \cref{thm:cancellation}]
Let $\pi$ be a finite group. It is shown in \cite{hambleton-kreck93}*{Theorem~B} that $\cb(\pi) \le 1$. Since finite groups are good \cite[p.~649]{Fr83}, \cref{thm:main-cancellation} now implies that $\cb^h(\pi) \le 1$ and $\cb^s(\pi) \le 1$, as required.
\end{proof}

We can also similarly obtain the following. Recall that a finitely presented group $\pi$ is called \textit{polycyclic-by-finite} if it has a subnormal series whose factors are finite or infinite cyclic $C_\infty$. The minimal number of copies of $C_\infty$ which arise in such a series is called the \textit{Hirsch length} of $\pi$ and is denoted by $H(\pi)$.

\begin{corollary}
    Let $\pi$ be a polycyclic-by-finite group with Hirsch length $H(\pi)$. Then 
    \[ \cb^h(\pi) \le \cb^s(\pi) \le H(\pi)+3.\]
In particular, let $M$, $N$ be closed oriented topological $4$-manifolds with the same Euler characteristic and polycyclic-by-finite fundamental group $\pi$ which are (simple) homotopy equivalent up to stabilisations. If $M\simeq M_0\# k(S^2 \times S^2)$ for a $4$-manifold $M_0$ where $k = H(\pi)+3$, then $M$, $N$ are (simple) homotopy equivalent.
\end{corollary}
    
\begin{proof}
It follows from \cite[Theorem 1.1]{CS} that $c(\pi) \le H(\pi)+3$. The result now follows from \cref{thm:main-cancellation} since polycyclic-by-finite groups are good \cite[p.~649]{Fr83}.
\end{proof}

\def\MR#1{}
\bibliographystyle{amsalpha}
\bibliography{classification}
\end{document}